\documentclass[12pt,reqno]{amsproc}%
\usepackage{amsfonts}
\usepackage{amsmath}
\usepackage{amssymb}
\usepackage{graphicx}%
 \theoremstyle{plain}
\newtheorem{theorem}{Theorem}[section]

\newtheorem{corollary}[theorem]{Corollary}

\newtheorem{definition}[theorem]{Definition}
\newtheorem{example}[theorem]{Example}

\newtheorem{lemma}[theorem]{Lemma}

\newtheorem{proposition}[theorem]{Proposition}
\newtheorem{remark}[theorem]{Remark}

\numberwithin{equation}  {section}

\begin{document}
\title[Beurling's theorem]{A non-commutative Beurling's theorem with respect to unitarily invariant norms}
\author{Yanni Chen}
\address{Department of Mathematics, University of New Hampshire, Durham, NH 03824, U.S.A.}
\email{yanni.chen@unh.edu}
\author{Don Hadwin}
\address{Department of Mathematics, University of New Hampshire, Durham, NH 03824, U.S.A.}
\email{don@unh.edu}
\author{Junhao Shen}
\address{Department of Mathematics, University of New Hampshire, Durham, NH 03824, U.S.A.}
\email{junhao.shen@unh.edu}
\thanks{This work supported in part by a grant from the Simons Foundation.}
\subjclass[2000]{Primary 46L52, 30H10; Secondary 47A15}
 \keywords{Normalized, unitarily invariant, $\Vert\cdot \Vert_{1}$-dominating,
continuous norm, maximal subdiagonal algebra, dual space, Beurling's
theorem, non-commutative Hardy space}

\begin{abstract}
 In 1967, Arveson invented a   non-commutative generalization
of classical $H^{\infty},$ known as finite  maximal subdiagonal subalgebras,
for a finite von Neumann algebra $\mathcal M$ with a faithful normal
tracial state $\tau$. In 2008, Blecher and Labuschagne proved a
version of Beurling's theorem on $H^\infty$-right invariant
subspaces in a non-commutative  $L^{p}(\mathcal M,\tau)$ space for
$1\le p\le \infty$. In the present paper, we define and study a
class of norms ${\mathcal{N}}_{c}(\mathcal M, \tau)$ on   $\mathcal{M},$ called
normalized, unitarily invariant,  $\Vert \cdot
\Vert_{1}$-dominating, continuous norms, which properly contains the
class $\{ \Vert \cdot \Vert_{p}:1\leq p< \infty \}.$ For $\alpha \in
\mathcal{N}_{c}(\mathcal M, \tau),$ we define a non-commutative  $L^{\alpha
}({\mathcal{M}},\tau)$ space and a non-commutative $H^{\alpha}$
space. Then we obtain a version of the Blecher-Labuschagne-Beurling
invariant subspace theorem on $H^\infty$-right invariant subspaces
in a non-commutative $L^{\alpha }({\mathcal{M}},\tau)$ space. Key
ingredients in the proof of our main result  include a
characterization theorem of  $H^\alpha$  and a density theorem for
$L^\alpha(\mathcal M,\tau)$.

\end{abstract}
\maketitle

\section{Introduction}

Let $\mathbb T$ be the unit circle and $\mu$ be the Haar measure on
$\mathbb T$ such that $\mu(\mathbb T)=1$. Then  $ L^\infty(\mathbb
T,\mu)$ is a commutative von Neumann algebra. For each $1\le p<
\infty$, we let $L^p(\mathbb T, \mu)$ be the completion of
$L^\infty(\mathbb T,\mu)$ with respect to $L^p$-norm. And we define
the Hardy space  $H^p$ as follows:
$$
H^p=\{f\in L^p(\mathbb T,\mu): \int_{\mathbb T} f(e^{i\theta})
e^{in\theta} d\mu(\theta)=0  \  \text{ for } \ n\in \mathbb N\},
\qquad \text{ for $1\le p \le \infty$.}
$$
It is not hard to check that,  for each $ 1\le p\le \infty $, there
exists a representation of $L^\infty(\mathbb T,\mu)$ into
$B(L^p(\mathbb T,\mu))$ given by the mapping $\psi \mapsto M_\psi$,
where $M_\psi$ is the multiplication operator defined by
$M_\psi(f)=\psi f$ for $f\in
 L^p(\mathbb T,\mu)$. Therefore we might assume that $L^\infty(\mathbb T,\mu)$,  and thus $H^\infty$, act naturally on each $L^p(\mathbb
T,\mu)$ space  by left (or right) multiplication for  $1\le p\le
\infty$. The classical, and influential, Beurling's  theorem in
\cite{B} states that {\em if
 $\mathcal W$ is a nonzero closed, $H^\infty$-invariant subspace (or, equivalently, $z\mathcal W\subseteq \mathcal W$) of $H^2$,
 then $\mathcal W=\psi H^2$ for some $\psi\in H^\infty$ with
 $|\psi|=1$ a.e.} $(\mu).$ Later, the Beurling's  theorem for $H^2$ was generalized to describe closed $H^\infty$-invariant subspaces in the Hardy space $H^p$
 with
 $1\le p\le \infty$   as follows: {\em if
 $\mathcal W$ is a nonzero closed $H^\infty$-invariant subspace of $H^p$
 with $1\le p\le \infty, $
 then $\mathcal W=\psi H^p$ for some $\psi\in H^\infty$ with
 $|\psi|=1$ a.e. $(\mu)$ } (see \cite{Bo}, \cite{Halmos}, \cite{He}, \cite{HL}, \cite{Ho}, \cite{Sr} and etc). The Beurling's  theorem was also extended in many other
 directions.

 The theory of non-commutative
$L^p$-spaces, or so called  ``non-commutative integration theory",
was initiated by  Segal (\cite{Se}) and Dixmier (\cite{Dix}) in
1950's. Since then, the theory of non-commutative $L^p$-spaces has
been extensively studied and developed (see \cite{PX} for related
references). It has now become an extremely active research area. In
the paper, we are mainly interested in non-commutative $L^p$-spaces
associated with   finite von Neumann algebras.  Let $\mathcal M$ be
a finite von Neumann algebra with a faithful normal tracial state
$\tau$. For each $1\le p<\infty$, we define a mapping $\|\cdot\|_p:
\mathcal M\rightarrow [0,\infty)$ by
$\|x\|_p=(\tau((x^*x)^{p/2}))^{1/p}$ for any $x\in \mathcal M$. It is
a highly nontrivial fact that $\|\cdot \|_p$ actually defines a
norm, an $L^p$-norm, on $\mathcal M$. Thus we let $L^p(\mathcal
M,\tau)$ be the completion of $\mathcal M$ under the norm
$\|\cdot\|_p$. Moreover, it is not hard to see that there exists an
anti-representation $\rho$ of $\mathcal M$ on the space
$L^p(\mathcal M,\tau)$ given by $\rho(a)\xi=\xi a$ for $\xi\in
L^p(\mathcal M,\tau)$ and $a\in \mathcal M$. Thus we might assume
that $\mathcal M $  acts naturally on each $L^p(\mathcal M,\tau)$
space  by   right  multiplication for $1\le p\le \infty$. We will
refer to a wonderful handbook \cite{PX} by Pisier and Xu for general
knowledge and current development of the theory of non commutative
$L^p$-spaces.

 In 1967, W. Arveson \cite{Arv}
introduced a concept of maximal subdiagonal algebras, also known as
non-commutative $H^\infty  $ spaces, to study the analyticity in
operator algebras. {   Let $\mathcal M$ be a finite von Neumann
algebra with a faithful normal tracial state $\tau$. Let  $\mathcal
A$ be a weak* closed unital subalgebra of $\mathcal M,$ and  $\Phi$
be a faithful, normal conditional  expectation from $\mathcal M$
onto a von Neumann subalgebra $\mathcal D $ of $\mathcal M$. Then
$\mathcal A$ is called a finite, maximal subdiagonal subalgebra of
$\mathcal M$ with respect to $\Phi$ if
 (i) $\mathcal A+\mathcal A^*$ is weak* dense in $\mathcal M;$
(ii) $\Phi(xy)=\Phi(x)\Phi(y)$ for all $x, y\in \mathcal A;$ (iii)
 $\tau\circ
\Phi=\tau;$ and (iv) $\mathcal D=\mathcal A\cap \mathcal A^*.$ (In
 \cite{Exel}, Excel showed that if $\mathcal{A}$ is weak* closed and $\tau$ satisfies (iii), then $\mathcal A$ (with respect to $\Phi$)
  is maximal among those subdiagonal subalgebras (with respect to
  $\Phi$)
satisfying (i), (ii) and (iv).)  Such a finite, maximal subdiagonal
subalgebra $\mathcal A$ of $\mathcal M$ is also called an $H^\infty$
space of $\mathcal M$. }  For each $1\le p< \infty$, we let $H^p$ be
the completion of Arveson's non-commutative $H^\infty$ with respect
to $\|\cdot\|_p.$

After Arveson's introduction of non-commutative $H^p$ spaces, there
are many studies  to obtain a Beurling's  theorem for invariant subspaces
in  non-commutative $H^p$ spaces (for example, see \cite{MMS}, \cite{Na}, \cite{NW}
and \cite{Sai}). It was Blecher and
  Labuschagne  who were able to show the following satisfactory version of
  Beurling's  theorem for $H^\infty$-invariant subspaces in a non-commutative $L^p(\mathcal M, \tau)$ space   in \cite{BL2}. {\em  Let $\mathcal M$ be a finite von Neumann
algebra with a faithful, tracial, normal state $\tau$, and $
H^{\infty}$ be a maximal subdiagonal subalgebra of $\mathcal{M}$
with $\mathcal D=H^\infty\cap (H^\infty)^*$. Suppose that $\mathcal
K$ is a closed $H^{\infty}$-right-invariant subspace of
$L^p(\mathcal{M},\tau),$ for some $1\leq p\leq \infty.$ (For
$p=\infty$ it is assumed that $\mathcal K$ is weak* closed.) Then
$\mathcal K$
may be written as a column $L^{p}$-sum $\mathcal K= \mathcal Z \bigoplus^{col}(\bigoplus^{col}%
_{i}u_{i}H^{p})$, where $ \mathcal Z $ is a closed (indeed weak*
closed if $p=\infty$)
  subspace of $L^{p}(\mathcal{M},\tau)$ such that $ \mathcal Z =[ \mathcal Z H_{0} ^{\infty}]_{p}$, and where $u_{i}$ are partial
isometries in ${\mathcal{M}}\cap \mathcal K$ satisfying certain
conditions (For more details, see \cite{BL2} or Lemma \ref{invariant on L^p}).} Here $\bigoplus^{col}%
_{i}u_{i}H^{p}$ and $ \mathcal Z =[ \mathcal Z H_{0} ^{\infty}]_{p}$
are of type 1, and  type 2 respectively (also see \cite{NW} for
definitions of invariant subspaces of different types).

 The concept of
unitarily invariant norms was introduced by von Neumann
\cite{vNeumann} for the purpose of metrizing matrix spaces. These
norms have now been   generalized and applied in many contexts (for
example, see \cite{Kunze}, \cite{McCarthy}, \cite{Simon} and etc).
Let $\mathcal M$ be a finite von Neumann algebra with a faithful
normal tracial state $\tau$. Besides all $L^p$-norms for $1\le p\le
\infty$, there are many other interesting examples of unitary
invariant norms on $\mathcal M$ (for example, see \cite{P},
\cite{P2}, \cite{Fang} and others). In the paper, we introduce a
class $N_c(\mathcal M,\tau)$ of normalized, unitarily invariant,
$\|\cdot\|_1$-dominating and continuous norms (see Definition
\ref{def2.2}), which properly contains all $L^p$-norms for $1\le
p<\infty$  and   many  unitarily invariant norms considered in
\cite{P}, \cite{P2} and \cite{Fang}. If $\alpha\in N_c(\mathcal
M,\tau)$ and $H^\infty$ is a finite, maximal subdiagonal subalgebra
of $\mathcal M$, then we let $L^\alpha(\mathcal M,\tau)$ and
$H^\alpha$ be the completion of $\mathcal M$, and $H^\infty$
respectively, with respect to the norm $\alpha$.  We also observe
that $\mathcal M$, and thus $H^\infty$, act naturally on
$L^\alpha(\mathcal M,\tau)$ by left, or right, multiplication (see
Lemma \ref{lemma2.3}). From Blecher and
  Labuschagne's result for non-commutative $H^p$ and $L^p(\mathcal M,\tau)$ spaces, it is
  natural to expect a Beurling's  theorem for $H^\alpha$ and $L^\alpha(\mathcal
  M,\tau)$ spaces.

In the paper, we consider  a version of Beurling's  theorem for
$H^\infty$-right invariant subspaces in $L^\alpha(\mathcal M,\tau)$,
and therefore for $H^\infty$-right invariant subspaces in
$H^\alpha$, when $\alpha\in N_c(\mathcal M,\tau)$. More
specifically,  we are able to obtain the following Beurling's  theorem
for $L^\alpha(\mathcal M ,\tau)$, built on Blecher and
  Labuschagne's result in the case of $p=\infty$.
 {
\renewcommand{\thetheorem}{\ref{mainthm}}
\begin{theorem}
Let $\mathcal M$ be a finite von Neumann algebra with a faithful,
normal, tracial state $\tau$. Let $H^\infty$  be a finite, maximal
subdiagonal subalgebra
  of $\mathcal M$ and $\mathcal D=H^\infty\cap (H^\infty)^*.$
 Let $\alpha$ be a normalized, unitarily
invariant, $\|\cdot \|_1$-dominating, continuous norm on $\mathcal
M$.

 If $\mathcal W$ is a   closed subspace of
$L^{\alpha}(\mathcal{M},\tau),$ then $\mathcal WH^{\infty}\subseteq
\mathcal W$ if and only if $$\mathcal W= \mathcal Z
\bigoplus^{col}(\bigoplus^{col}_{i\in \mathcal I}u_{i}H^{\alpha}),$$
where $ \mathcal Z $ is a closed    subspace of
$L^{\alpha}(\mathcal{M},\tau)$ such that $ \mathcal Z =[ \mathcal Z
H_{0} ^{\infty}]_{\alpha}$ , and where $u_{i}$ are partial
isometries in $\mathcal W  \cap  {\mathcal{M}}$ with
$u_{j}^{*}u_{i}=0$
if $i\neq j,$ and with $u_{i}^{*}u_{i}\in{\mathcal{D}}.$ Moreover, for each $i,$ $u_{i}%
^{*} \mathcal Z =\{0\},$ left multiplication by the $u_{i}u_{i}^{*}$
are contractive projections from $\mathcal W$ onto the summands
$u_{i} H^{\alpha},$ and left multiplication by
$1-\sum_{i}u_{i}u_{i}^{*}$ is a contractive projection from
$\mathcal W$ onto $ \mathcal Z .$
\addtocounter{theorem}{-1}
\end{theorem}
\addtocounter{theorem}{-1}
} Here $\bigoplus^{col}$ denotes an internal column sum (see Definition  \ref{def5.5}). Moreover,  $\bigoplus^{col}%
_{i}u_{i}H^{\alpha}$ and $ \mathcal Z =[ \mathcal Z H_{0}
^{\infty}]_{\alpha}$ are of type 1, and  of type 2 respectively (see
\cite{NW}, \cite{BL2} for definitions of invariant subspaces of
different types).

 Many tools used in a non-commutative $L^p(\mathcal
M,\tau)$ space are no longer available in an arbitrary
$L^\alpha(\mathcal M,\tau)$ space and new techniques or new proofs
need to be invented. Key ingredients in the proof of Theorem
{\ref{mainthm}}  include a characterization of  $H^\alpha$ (see
Theorem \ref{characterization of H^alpha}), a factorization result
in $L^\alpha(\mathcal M,\tau)$  (see Proposition \ref {decomposition
in L^alpha}), and a density theorem for $L^\alpha(\mathcal M,\tau)$
(see Theorem \ref{invariant on L^alpha}), which extend earlier
results by Saito in \cite{Sai}.

{\renewcommand{\thetheorem}{\ref{characterization of H^alpha}}
\begin{theorem}
 Let $\mathcal M$ be a finite von Neumann algebra with a faithful
normal tracial state $\tau$, and $H^\infty$ be a  finite, maximal
subdiagonal subalgebra
  of $\mathcal M$. Let $\alpha$ be a normalized, unitarily invariant, $\|\cdot \|_1$-dominating, continuous norm on $\mathcal M$.  Then
\[
H^{\alpha}=H^{1}\cap L^{\alpha}(\mathcal{M},\tau)=\{x\in
L^{\alpha}(\mathcal{M},\tau): \tau(xy)=0 \  \mbox{for all } y\in
H_{0}^{\infty} \}.
\]

\addtocounter{theorem}{-1}
\end{theorem}
\addtocounter{theorem}{-1}}

{ {\renewcommand{\thetheorem}{\ref{decomposition in L^alpha}}
\begin{proposition}  Let $\mathcal M$ be a finite von
Neumann algebra with a faithful normal tracial state $\tau $, and
$H^\infty$  be a finite, maximal subdiagonal subalgebra
  of $\mathcal M$.  Let
$\alpha$ be a normalized, unitarily invariant, $\|\cdot
\|_1$-dominating, continuous norm on $\mathcal M$. If $k\in
\mathcal{M}$ and $k^{-1}\in L^{\alpha}(\mathcal{M},\tau),$ then
there are unitary operators $w_{1}, w_{2}\in \mathcal{M}$ and
operators $a_{1}, a_{2}\in H^{\infty}$ such that
$k=w_{1}a_{1}=a_{2}w_{2}$ and $a_{1}^{-1}, a_{2}^{-1}\in
H^{\alpha}.$ \addtocounter{theorem}{-1}
\end{proposition}
\addtocounter{theorem}{-1}}

{\renewcommand{\thetheorem}{\ref{invariant on L^alpha}}
\begin{theorem}  Let $\mathcal M$ be a finite von Neumann algebra with a faithful
normal tracial state $\tau$, and $H^\infty$  be a finite, maximal
subdiagonal subalgebra
  of $\mathcal M$. Let $\alpha$ be a normalized, unitarily invariant, $\|\cdot \|_1$-dominating, continuous norm on $\mathcal M$. % (see
%Definition \ref{def2.2}).  Let  $L_{\overline {\alpha'}}(\mathcal
%M,\tau)$ be as defined {in Definition \ref{def2.10}}.

If $\mathcal W$ is a closed subspace of $L^\alpha(\mathcal M,\tau) $
and $\mathcal N$ is a weak*-closed linear subspace of $\mathcal M$
such that $\mathcal  W H^\infty\subseteq \mathcal W$ and $  \mathcal
N H^\infty\subseteq \mathcal N,$  then

\begin{enumerate}
\item $\mathcal N=[\mathcal N]_{\alpha}\cap \mathcal M;$

\item $\mathcal W\cap\mathcal M $ is weak* closed in $\mathcal M;$

\item $\mathcal W=[  \mathcal W\cap\mathcal M]_{\alpha};$

\item if $\mathcal S$ is a subspace of $\mathcal M$ such that $\mathcal S H^\infty\subseteq \mathcal S$, then $$[\mathcal
S]_\alpha= [\overline {\mathcal S}^{w*}]_\alpha,$$ where $\overline
{\mathcal S}^{w*}$ is the weak*-closure of $\mathcal S$ in $\mathcal
M$.
\end{enumerate}
\addtocounter{theorem}{-1}
\end{theorem}
\addtocounter{theorem}{-1}}

We end the paper with two quick applications of Theorem
{\ref{mainthm}}, which contain classical Beurling's  theorem as a
special case by letting $\mathcal M$   be $L^\infty(\mathbb T,\mu)$.

{\renewcommand{\thetheorem}{\ref{Cor5.8}}
\begin{corollary}
Let $\mathcal M$ be a finite von Neumann algebra with a faithful,
normal, tracial state $\tau$.
 Let $\alpha$ be a normalized, unitarily
invariant, $\|\cdot \|_1$-dominating, continuous norm on $\mathcal
M$. If $\mathcal W$ is a   closed subspace of
$L^{\alpha}(\mathcal{M},\tau) $ such that $\mathcal W\mathcal
M\subseteq \mathcal W$, then there exists a projection $e$ in
$\mathcal M$ such that $\mathcal W= eL^\alpha(\mathcal M,\tau)$.

\addtocounter{theorem}{-1}
\end{corollary}
\addtocounter{theorem}{-1}}

{\renewcommand{\thetheorem}{\ref{Cor5.9}}
\begin{corollary}
Let $\mathcal M$ be a finite von Neumann algebra with a faithful,
normal, tracial state $\tau$. Let $H^\infty$  be a finite, maximal
subdiagonal subalgebra
  of $\mathcal M$  such that $   H^\infty\cap (H^\infty)^*=\mathbb CI.$
 Let $\alpha$ be a normalized, unitarily
invariant, $\|\cdot \|_1$-dominating, continuous norm on $\mathcal
M$.

 Assume that $\mathcal W$ is a   closed subspace of
$L^{\alpha}(\mathcal{M},\tau).$  If $\mathcal W$ is  simply
$H^\infty$-right invariant, i.e. $\mathcal WH^{\infty}\subsetneqq
\mathcal W$, then there exists a unitary $u\in \mathcal W\cap
\mathcal M$ such that $\mathcal W=uH^\alpha$.

\addtocounter{theorem}{-1}
\end{corollary}
\addtocounter{theorem}{-1}}

The organization of  the paper is as follows. In section 2, we
introduce a class $N_c(\mathcal M,\tau)$ of normalized, unitarily
invariant, $\|\cdot\|_1$-dominating and continuous norms and study
their dual norms on a finite von Neumann algebra $\mathcal M$ with a
faithful normal tracial state $\tau$. In section 3, we prove a
H\"{o}lder's inequality and use it to find the dual space of
$L^{\alpha}(\mathcal{M},\tau)$ when $\alpha\in N_c(\mathcal
M,\tau)$.  In Section 4, we define the non-commutative $H^{\alpha}$
spaces  and provide a characterization of $H^\alpha$. In section 5,
we prove the main result of the paper, a version of Beurling's
theorem for $H^\infty$-right invariant subspaces in
$L^{\alpha}(\mathcal M,\tau)$ spaces.

\section{Unitarily invariant norms and dual norms on finite von Neumann algebras}

\subsection{Unitarily invariant norms} Let $\mathcal M$ be a finite von Neumann algebra with a faithful
normal tracial state $\tau$. For general knowledge about
non-commutative $L^p$-spaces for $0< p\le \infty$ associated with
a von Neumann algebra $\mathcal M$, we will refer to a wonderful
handbook \cite{PX} by Pisier and Xu. For each $0< p<\infty$, we
let $\| \cdot\|_p$ be a mapping from $\mathcal M$ to $[0,\infty)$ (see
\cite{PX}) as defined by
$$
\|x\|_p=\left ( \tau(|x|^p)\right)^{1/p}, \qquad \forall \ x\in
\mathcal M.
$$
It is known that $\|\cdot \|_p$ is a norm   if $1\le p<\infty$, and a quasi-norm  if $0<p<1$.
We define $L^p(\mathcal M, \tau)$, so called non-commutative
$L^p$-space associated with $(\mathcal M,\tau)$, to be the
completion of $\mathcal M$ with respect to
$\|\cdot\|_p$ for $0< p<\infty$.

In the paper, we will mainly focus on the following two classes of
unitarily invariant norms of a finite von Neumann algebra.

\begin{definition}\label{def2.1}
We denote by $N(\mathcal M,\tau)$ the collection of all these norms
$\alpha :\mathcal M\rightarrow [0,\infty)$ satisfying:

\begin{enumerate}
\item [(a)] $\alpha(I)=1,$ i.e. $\alpha$ is normalized.

\item  [(b)]$\alpha(uxv)=\alpha(x)$ for all $x\in \mathcal{M} $ and  unitaries  $u,v$  in $\mathcal{M},$  i.e. $\alpha $ is unitarily invariant.

\item  [(c)]$ \Vert x\Vert_{1}\le \alpha(x)$ for every $x\in \mathcal{M},$
i.e. $\alpha$ is $\|\cdot \|_1$-dominating.
\end{enumerate}
A norm $\alpha$ in $N(\mathcal M,\tau)$ is called a
\emph{normalized, unitarily invariant, $\|\cdot \|_1$-dominating
norm} on $\mathcal M$.
\end{definition}
\begin{definition}\label{def2.2}
We denote by $N_c(\mathcal M,\tau)$ the collection of all these
norms $\alpha :\mathcal M\rightarrow [0,\infty)$  such that
\begin{enumerate}
\item [(a)] $\alpha\in N(\mathcal M,\tau)$ and
\item [(b)]
$\displaystyle \lim_{\tau(e)\rightarrow 0}\alpha(e)=0 $ as $e$
ranges over the projections in $\mathcal{M}$ ($\alpha$ is a
continuous norm with respect to a trace $\tau$).
\end{enumerate}
A norm $\alpha$ in $N_c(\mathcal M,\tau)$ is called a
\emph{normalized, unitarily invariant, $\|\cdot \|_1$-dominating,
continuous norm} on $\mathcal M$.
\end{definition}
\begin{example}
Each $p$-norm, $\|\cdot\|_p$, is in the class $N_c(\mathcal M,\tau)$
for $1\le p<\infty$.
\end{example}

\begin{example}
Let $\mathcal M$ be a finite von Neumann algebra with a faithful normal tracial state $\tau$ satisfying the weak Dixmier property  (See \cite{Fang}). Let $\alpha$ be a normalized tracial gauge norm on $\mathcal M$. Then Theorem 3.30 in \cite{Fang} shows that  $\alpha\in N(\mathcal M,\tau)$.
\end{example}

\begin{example}
Let $\mathcal M$ be a finite  von Neumann algebra with a faithful
normal tracial state $\tau$ and $E(0,1)$ be a rearrangement
invariant symmetric Banach function space on $(0,1)$. A
non-commutative Banach function space $E(\tau)$ together with a norm
$\|\cdot\|_{E(\tau)}$, corresponding to $E(0,1)$ and associated with
$(\mathcal M,\tau)$, can be introduced (see \cite{P} or \cite{P2} ).
Moreover $\mathcal M$ is a  subset in $E(\tau)$  and the restriction
of the norm $\|\cdot\|_{E(\tau)}$ to $\mathcal M$ lies in
$N({\mathcal M,\tau})$. If  $E $ is also order continuous, then the
restriction of the norm $\|\cdot\|_{E(\tau)}$ to $\mathcal M$ lies
in $N_c({\mathcal M,\tau})$.
\end{example}

\begin{example}
Let $\mathcal N$ be a type II$_1$ factor with a tracial state $\tau_{\mathcal N}$. Let $\|\cdot\|_{1,\mathcal N}$ and $\|\cdot\|_{2,\mathcal N}$ be $L^1$-norm, and $L^2$-norm respectively, on $\mathcal N$. Let $\mathcal M=\mathcal N\oplus \mathcal N$ be a finite von Neumann algebra with a faithful normal tracial state $\tau$, defined by
$$
\tau(x\oplus y)=\frac {\tau_{\mathcal N}(x)+\tau_{\mathcal N}(y)}2, \qquad \forall \ x\oplus y\in \mathcal M.
$$ Let $\alpha$ be a norm of $\mathcal M$, defined by
$$
\alpha(x\oplus y)=\frac {\|x\|_{1,\mathcal N} +\|y\|_{2,\mathcal N}}2, \qquad \forall \ x\oplus y\in \mathcal M.
$$
Then $\alpha\in N_c(\mathcal M,\tau)$.  But $\alpha$ is neither
tracial (see Definition 3.7 in \cite{Fang}) nor rearrangement
invariant (see Definition 2.1 in \cite{Do}).
\end{example}

The following lemma is well-known.
\begin{lemma}\label{lemma2.3}
Let $\mathcal M$ be a finite  von Neumann algebra with a faithful
normal tracial state $\tau$ and $\alpha$ be a norm on $\mathcal M$.
If $\alpha$ is unitarily invariant, i.e.
$$\text{$\alpha(uxv)=\alpha(x)$ for all $x\in \mathcal{M} $ and  unitaries  $u,v$  in $\mathcal{M}$,}$$
then
$$
\alpha(x_1yx_2)\le \|x_1\|\cdot \|x_2\|\cdot \alpha(y), \qquad
\forall \ x_1,x_2, y \in\mathcal M.
$$ In particular, if $\alpha$ is a normalized unitarily invariant norm on $\mathcal M$, then
 $$\alpha(x)\le \|x\|, \qquad \forall \ x\in\mathcal M.$$
\end{lemma}
\begin{proof}
Let $x\in \mathcal M$ such that $\|x\|=1.$ Assume that $x=v|x|$ is
the polar decomposition of $x$ in $\mathcal M$, where $v$ is a
unitary in $\mathcal M$ and $|x|$ in $\mathcal M$ is positive. Then
$u=|x|+i\sqrt {I-|x|^2}$ is a unitary in $\mathcal M$ such that
$|x|=(u+u^*)/2$. Thus $$\alpha(xy)=\alpha(|x|y)= \alpha(\frac
{uy+u^*y} 2)\le \frac {\alpha(uy)+\alpha(u^*y)} 2= \alpha(y).$$
Hence $\alpha(xy)\le \|x\|\alpha(y),  \forall \ x, y \in\mathcal M.$
Similarly, $\alpha( yx)\le \|x\|\alpha(y),  \forall \ x, y
\in\mathcal M.$

Furthermore, if   $\alpha$ is a normalized unitarily invariant norm on $\mathcal M$, then  from the discussion in the preceding paragraph  we have that
$$
\alpha(x)\le \|x\|\alpha(I)= \|x\|, \qquad  \forall \ x\in\mathcal M.
$$
\end{proof}

\subsection{Dual norms of unitarily invariant norms on $\mathcal{M}$.}

The concept of dual norm plays an important role in the study of
non-commutative $L^p$-spaces. In this subsection, we will introduce
dual norm for a unitarily invariant norm on a finite von Neumann
algebra.

\begin{lemma} \label{lemma2.6}
Let $\mathcal M$ be a finite von Neumann algebra with a faithful
normal tracial state $\tau$. Let $\alpha$ be a normalized, unitarily
invariant, $\|\cdot \|_1$-dominating norm on $\mathcal M$ (see
Definition \ref{def2.1}).  Define a mapping $\alpha':\mathcal
M\rightarrow [0,\infty]$ as follows:
\[
\alpha^{\prime}(x)=\sup \{|\tau(xy)|: y\in{\mathcal{M}},
\alpha(y)\leq1\}, \qquad \forall \ x\in \mathcal M.
\]
Then the following statements are true.
\begin{enumerate}
\item [(i)] $\forall \ x\in \mathcal M, \|x\|_1\le \alpha'(x) \le \|x\|.$
\item [(ii)] $\alpha'$ is a norm on $\mathcal M$.
\item [(iii)] $\alpha'\in N(\mathcal M, \tau)$, i.e. $\alpha'$ is a normalized, unitarily invariant, $\|\cdot \|_1$-dominating norm.
\item [(iv)] $|\tau(xy)|\le \alpha(x)\alpha'(y)$ for all $x,y$ in $\mathcal M$.
\end{enumerate}

\end{lemma}

\begin{proof}
(i) Suppose $x\in \mathcal{M}.$ If $y\in \mathcal{M}$ with
$\alpha(y)\leq1,$ then, from the fact that $\alpha$ is
$\|\cdot\|_1$-dominating, we have
\[
|\tau(xy)|\leq \|x\| \|y\|_{1}\leq \|x\| \alpha(y)\leq \|x\|,
\]
whence $\alpha^{\prime}(x)\leq \|x\|.$ Thus $\alpha'$ is a mapping
from $\mathcal M$ to $[0,\infty)$.

Now, assume that $x=uh$ is the polar decomposition of $x$ in
$\mathcal M$, where $u$ is a unitary element in $\mathcal M$ and $h$
in $\mathcal M$ is positive. Then, from the fact that $\alpha(u^*)=1$,
we have
$$\alpha^{\prime}(x)\geq |\tau(u^{*}x)|=  \tau(h)=\|x\|_{1}.$$ Therefore $\|x\|_{1}\leq \alpha^{\prime}%
(x)$ for every $x\in \mathcal{M}.$  This ends the proof of part (i).

 (ii) It is easy to verify that
$$
\alpha^\prime(ax)=|a|\alpha'(x), \ \text { and } \ \alpha'(x_1+x_2)\le \alpha'(x_1)+\alpha'(x_2), \qquad \forall a\in\mathbb C, \forall \ x,
x_1,x_2\in \mathcal M.
$$
From the result (i), we know that $\alpha'(x)=0$ implies $x=0$.
Therefore  $\alpha'$ is a norm on $\mathcal M$.

(iii) It is not hard to verify that $\alpha'$ satisfies conditions
(a) and (b) in the definition of $N(\mathcal M,\tau)$. From the
result (i), $\alpha'$ also satisfies condition  (c)  in the
definition of $N(\mathcal M,\tau)$. Therefore $\alpha'\in N(\mathcal
M,\tau)$.

(iv) It follows directly from the definition of $\alpha'$.
\end{proof}

\begin{definition} \label{def2.8}
The  $\alpha'$, as defined in Lemma \ref{lemma2.6},   is called
the \emph{dual norm} of $\alpha$ on $\mathcal M$.
\end{definition}
Now we are ready to introduce $L^\alpha$-space and
$L^{\alpha'}$-space for a finite von Neumann algebra $\mathcal M$
with respect to the unitarily invariant norms $\alpha$, and
$\alpha'$ respectively, as follows.
\begin{definition}Let $\mathcal M$ be a finite von Neumann algebra with a faithful
normal tracial state $\tau$. Let $\alpha$ be a normalized, unitarily
invariant, $\|\cdot \|_1$-dominating norm on $\mathcal M$ (see
Definition \ref{def2.1}). Let $\alpha'$ be the dual norm of $\alpha$
on $\mathcal M$ (see Definition \ref{def2.8}).
  We
define $L^\alpha(\mathcal M, \tau)$ and $L^{\alpha'}(\mathcal M,
\tau)$   to be the completion of $\mathcal M$ with respect to
  $\alpha$, and $\alpha'$ respectively.
\end{definition}

\begin{remark}
If $\alpha$ is an $L^p$-norm  for some $1<p<\infty$, then $\alpha'$
is nothing but an $L^q$-norm  where $1/p+1/q=1$. Hence
$L^\alpha(\mathcal M, \tau)$, $L^{\alpha'}(\mathcal M, \tau)$ are
the usual $L^p(\mathcal M, \tau)$,  $L^q(\mathcal M, \tau)$.

It is known that the dual space of $L^p(\mathcal M, \tau)$  is
$L^q(\mathcal M, \tau)$  when $1<p,q<\infty$ and  $1/p+1/q=1$. However generally, for
  $\alpha\in N(\mathcal M,\tau)$, the dual of $L^\alpha(\mathcal
M, \tau)$ might not be $L^{\alpha'}(\mathcal M, \tau)$. %(See \cite{Do})
\end{remark}

\section{Dual spaces of  $L^{\alpha}$-spaces associated with finite von Neumann algebras}

In this section we will study dual space of
$L^{\alpha}({\mathcal{M}},\tau)$ by investigating some subspaces in $L^1(\mathcal M,\tau)$.

\subsection{Definitions of subspaces $L_{\overline {\alpha}}(\mathcal M,\tau) $ and $L_{\overline {\alpha'}}(\mathcal M,\tau) $ of $L^1(\mathcal M,\tau)$}

\begin{definition}\label{def2.10}
Let $\mathcal M$ be a finite von Neumann algebra with a faithful
normal tracial state $\tau$. Let $\alpha$ be a normalized, unitarily
invariant, $\|\cdot \|_1$-dominating norm on $\mathcal M$ (see
Definition \ref{def2.1}). Let $\alpha'$ be the dual norm of $\alpha$
on $\mathcal M$ (see Definition \ref{def2.8}).

We define   $$\overline {\alpha}: L^1(\mathcal M, \tau)\rightarrow [0,\infty]  \quad \text { and } \quad \overline {\alpha'}: L^1(\mathcal M, \tau)\rightarrow [0,\infty]$$
as follows:
$$
\begin{aligned}
\overline {\alpha}(x)&=\sup \{|\tau(xy)|: y\in{\mathcal{M}}, \alpha^{\prime}(y)\leq1\}, \qquad \text{$\forall \ x\in   L^1(\mathcal M, \tau)$, }
 \\
\overline {\alpha'}(x)&=\sup \{|\tau(xy)|: y\in{\mathcal{M}}, \alpha(y)\leq1\},   \qquad \text{$\forall \ x\in   L^1(\mathcal M, \tau)$. }
\end{aligned}
$$
We define
$$
\begin{aligned}
L_{\overline {\alpha}}(\mathcal M,\tau) &=  \{  x\in   L^1(\mathcal M, \tau): \overline {\alpha}(x)<\infty \} \subseteq L^1(\mathcal M,\tau)
 \\
L_{\overline {\alpha'}}(\mathcal M,\tau) &=  \{  x\in   L^1(\mathcal M, \tau): \overline {\alpha'}(x)<\infty \} \subseteq L^1(\mathcal M,\tau).
\end{aligned}
$$
\end{definition}

Thus  $\overline {\alpha}$ and $\overline {\alpha'}$ are mappings from $L_{\overline {\alpha}}(\mathcal M,\tau)$, and $L_{\overline {\alpha'}}(\mathcal M,\tau)$ respectively, into $[0,\infty)$.
The next result follows directly from the definitions of $\overline {\alpha}$, $\overline {\alpha'}$ and part (iv) of Lemma \ref{lemma2.6}.
 \begin{lemma} \label{lemma3.2} We have
 \[
\overline {\alpha'}(x)=\alpha^{\prime}(x)  \  \  \text{ and }  \ \ \overline {\alpha}(x) \leq \alpha(x) \ \ \ \mbox{\ for\  every \ } x\in \mathcal{M}.
\]
 \end{lemma}

The following proposition describes   properties of $\overline {\alpha}$ and $\overline {\alpha'}$, which imply that $L_{\overline {\alpha}}(\mathcal M,\tau) $ and $L_{\overline {\alpha'}}(\mathcal M,\tau) $ are normed spaces with respect to $\overline {\alpha}$ and $\overline {\alpha'}$ respectively.

\begin{proposition}
\label{property of beta} Let $\mathcal M$ be a finite von Neumann
algebra with a faithful normal tracial state $\tau$. Let $\alpha$ be
a normalized, unitarily invariant, $\|\cdot \|_1$-dominating norm on
$\mathcal M$ (see Definition \ref{def2.1}). Let $\alpha'$ be the
dual norm of $\alpha$ on $\mathcal M$ (see Definition \ref{def2.8}).
Let $$\overline {\alpha}: L_{\overline {\alpha}}(\mathcal M,\tau)
\rightarrow [0,\infty)
  \quad \text { and } \quad \overline {\alpha'}: L_{\overline {\alpha'}}(\mathcal M,\tau)\rightarrow [0,\infty)$$be as in Definition \ref{def2.10}.

  Then
the following statements are true:

\begin{enumerate}
\item [(i)] $\overline {\alpha}(I)=1$ and $\overline {\alpha'}(I)=1.$
\item  [(ii)] If  $u,v$ are unitary elements in $\mathcal M$, then
$$
\overline {\alpha}(x)=\overline {\alpha}(uxv), \qquad \forall \ x\in    L_{\overline {\alpha }}(\mathcal M,\tau) $${ and }$$ \overline {\alpha'}(x)=\overline {\alpha'}(uxv),  \qquad \forall \ x\in  L_{\overline {\alpha'}}(\mathcal M,\tau).
$$
\item  [(iii$_1$)] We have
$$
\|x\|_1\le \overline {\alpha}(x),  \qquad \forall \ x\in    L_{\overline {\alpha }}(\mathcal M,\tau) $${ and }$$ \|x\|_1\le\overline {\alpha'}(x),  \qquad \forall \ x\in  L_{\overline {\alpha'}}(\mathcal M,\tau).
$$
\item [(iii$_2$)]
If $x$ is an element in $\mathcal M$, then
$$
  \overline {\alpha}(x)\le \|x\|  \qquad \text{ and } \qquad    \overline {\alpha'}(x) \le \|x\| .
$$

 \item [(iv)] $ \overline {\alpha}$ and $ \overline {\alpha'}$ are norms on $L_{\overline {\alpha }}(\mathcal M,\tau) $, and $L_{\overline {\alpha' }}(\mathcal M,\tau) $ respectively.
\end{enumerate}
\end{proposition}

\begin{proof}
(i) Note that $\alpha \in N(\mathcal M,
\tau)$ and $\alpha'\in N(\mathcal M,
\tau)$ from part (iii) of Lemma \ref{lemma2.6} . Thus
$$
\overline {\alpha}(I)=\sup \{|\tau(y)|: y\in{\mathcal{M}}, \alpha^{\prime}(y)\leq1\} =\sup \{||y||_1: y\in{\mathcal{M}}, \alpha^{\prime}(y)\leq1\}=1.
$$Similarly,
$$
\overline {\alpha'}(I)=1.
$$

(ii)  If $u,v$ are unitaries
in $\mathcal{M},$ then
\begin{align}
\overline {\alpha}(uxv) & =\sup \{|\tau(uxvy)|: y\in{\mathcal{M}}, \alpha^{\prime}%
(y)\leq1\} \tag{by Definition \ref{def2.10}}\\
& =\sup \{|\tau(xvyu)|: y\in{\mathcal{M}}, \alpha^{\prime}(y)\leq1\}\notag \\
& =\sup \{|\tau(xy_{0})|: y\in{\mathcal{M}}, \alpha^{\prime}(y_{0}%
)=\alpha^{\prime}(vyu)=\alpha^{\prime}(y)\leq1\} \tag{because $\alpha'\in N(\mathcal M, \tau)$} \\
& =\overline {\alpha}(x),   \text{ $\forall x\in L_{\overline {\alpha}}(\mathcal M,\tau).$} \notag
\end{align}
Similarly, we have
$$ \overline {\alpha'}(x)=\overline {\alpha'}(uxv),  \qquad \forall \ x\in  L_{\overline {\alpha'}}(\mathcal M,\tau).
$$

(iii$_1$) Assume that $x\in  L_{\overline {\alpha }}(\mathcal
M,\tau)\subseteq L^1(\mathcal M,\tau)$. We  let $x=uh$ be the polar
decomposition of $x$ in $L^{1}(\mathcal{M}),$ where $u$ is a unitary
in $\mathcal{M}$ and $h=|x|\in L^{1}(\mathcal{M}).$ Then, from the
result (ii), we obtain that
\[
\overline {\alpha }(x)=\overline {\alpha }(uh)=\overline {\alpha }(h)\geq|\tau(h)|=\|x\|_{1},
\]
Similarly, we have
$$ \|x\|_1\le\overline {\alpha'}(x),  \qquad \forall \ x\in  L_{\overline {\alpha'}}(\mathcal M,\tau).
$$

(iii$_2$)  Note that  $\alpha'\in N(\mathcal M,
\tau)$.
Suppose $x\in \mathcal{M}.$ If $y\in \mathcal{M}$ with $\alpha^{\prime}%
(y)\leq1,$ then
\[
|\tau(xy)|\leq \|x\| \|y\|_{1}\leq \|x\| \alpha^{\prime}(y)\leq \|x\|.
\]
Now it follows from the definition of $\overline{\alpha}$ that $\overline{\alpha}(x)\leq \|x\|$. Similarly, we have  $\overline{\alpha'}(x) \le \|x\|$,
$\forall x\in \mathcal{M}.$

(iv) From the definition and the result (iii$_1$), we conclude that  $ \overline {\alpha}$ and $ \overline {\alpha'}$ are norms on $L_{\overline {\alpha }}(\mathcal M,\tau) $, and $L_{\overline {\alpha' }}(\mathcal M,\tau) $ respectively.

\end{proof}
The following lemma is a useful tool for our later results.
\begin{lemma}
\label{property of beta2} Let $\mathcal M$ be a finite von Neumann
algebra with a faithful normal tracial state $\tau$. Let $\alpha$ be
a normalized, unitarily invariant, $\|\cdot \|_1$-dominating norm on
$\mathcal M$ (see Definition \ref{def2.1}). Let $\alpha'$ be the
dual norm of $\alpha$ on $\mathcal M$ (see Definition \ref{def2.8}).
Let $ \overline {\alpha} $ and $\overline {\alpha'} $ be as in
Definition \ref{def2.10}. Then the following statements are true.

\begin{enumerate}
\item [(i)] For all $a\in \mathcal{M}$ and $x\in L_{\overline {\alpha}}(\mathcal M,\tau),$
$\overline {\alpha}(ax)\leq \|a\| \overline {\alpha}(x).$

\item [(ii)]  For all $a\in \mathcal{M}$ and $x\in L_{\overline {\alpha'}}(\mathcal M,\tau),$
$\overline {\alpha'}(ax)\leq \|a\|\overline {\alpha'}(x).$
\end{enumerate}
\end{lemma}

\begin{proof}
(i) From Proposition \ref{property of beta},  $\overline{\alpha}$ is a norm on  $L_{\overline {\alpha}}(\mathcal M,\tau)$ satisfying
$$
\overline {\alpha}(x)=\overline {\alpha}(uxv), \qquad \forall \text {     unitary elements $u,v \in \mathcal M$ and }   \ x\in    L_{\overline {\alpha }}(\mathcal M,\tau)  .
$$Now the   proof of Lemma \ref{lemma2.3} can also be applied here.

 % Suppose $a\in \mathcal{M}$ with $\|a\|<1$ and $x\in L_{\overline {\alpha}}(\mathcal M,\tau)$.  Then, from \ \ , there are $n$ unitaries
%$u_{1}, u_{2}, \cdots, u_{n}\in \mathcal{M}$ such that $$a=\frac{u_{1}%
%+u_{2}+\cdots u_{n}}{n}.$$ It follows from part (ii) of Proposition
%\ref{property of beta} that
%\[
%\overline {\alpha}(ax)=\overline {\alpha}(\frac{u_{1}x+u_{2}x+\cdots u_{n}x}{n})\leq \frac{\overline {\alpha}
%(u_{1}x)+\overline {\alpha}(u_{2}x)+\cdots\overline {\alpha}(u_{n}x)}{n}=\overline {\alpha}(x).
%\]
%Now  for every $a\in \mathcal{M} $ and  $\epsilon>0$, we have
%\[
%\overline {\alpha}(\frac{a}{\|a\|+\epsilon}x)\leq \overline {\alpha}(x).
%\]
%Since $\epsilon$ is arbitrary, we have  $\overline {\alpha}(ax)\leq \|a\| \overline {\alpha}(x).$

(ii)  Similar  result holds for $\overline {\alpha'}.$
\end{proof}

Our next result shows that $L_{\overline {\alpha}}(\mathcal M,\tau) $ and $L_{\overline {\alpha'}}(\mathcal M,\tau) $ are Banach spaces with respect to $\overline {\alpha}$ and $\overline {\alpha'}$ respectively.
\begin{proposition}
 Let $\mathcal M$ be a finite von Neumann algebra with a faithful
normal tracial state $\tau$. Let $\alpha$ be a normalized, unitarily
invariant, $\|\cdot \|_1$-dominating  norm on $\mathcal M$ (see
Definition \ref{def2.1}). Let $\alpha'$ be the dual norm of $\alpha$
on $\mathcal M$ (see Definition \ref{def2.8}).   Let $ \overline
{\alpha} $, $\overline {\alpha'} $, $L_{\overline {\alpha}}(\mathcal
M,\tau)$, and $L_{\overline {\alpha'}}(\mathcal M,\tau)$ be as in
Definition \ref{def2.10}.

Then  $L_{\overline {\alpha}}(\mathcal M,\tau)$ and $L_{\overline {\alpha'}}(\mathcal M,\tau)$ are both Banach spaces with respect to norms $\overline{\alpha}$ and $\overline{\alpha'}$ respectively.
\end{proposition}

\begin{proof} Since arguments for   $L_{\overline {\alpha}}(\mathcal M,\tau)$ and    for $L_{\overline {\alpha'}}(\mathcal M,\tau)$  are similar, we will only present the proof that $L_{\overline {\alpha}}(\mathcal M,\tau)$ is a  Banach space here.

From part (iv) of Proposition \ref{property of beta}, we know that  $L_{\overline {\alpha}}(\mathcal M,\tau)$ is a  normed space with respect to  $\overline{\alpha}$.  To prove the completeness of the space, we suppose $\{x_{n}\}$ is a Cauchy
sequence in  $L_{\overline {\alpha}}(\mathcal M,\tau)$ with respect to  $\overline{\alpha}$. Then there is an $M>0$ such
that $\overline {\alpha}(x_{n})\leq M$ for all $n.$ From part (iii$_1$) of Proposition \ref{property of beta}, we have that  $\|x_m-x_{n}\|_{1}\leq \overline {\alpha}(x_m-x_{n}) $ for $m,n\ge 1$.
It follows that $\{x_{n}\}$ is a Cauchy sequence in $L^{1}(\mathcal{M},\tau),$
which is a complete Banach space. Then there is an $x_{0}\in L^{1}(\mathcal{M},\tau)$ such that
$\|x_{n}-x_{0}\|_{1}\rightarrow0.$

We claim that $x_0\in L_{\overline {\alpha}}(\mathcal M,\tau)$ and $ \overline {\alpha}(x_n-x_0)\rightarrow 0$ as $n$ goes to infinity. In fact, we let $y\in \mathcal{M }$ with $  {\alpha'}(y)\leq1.$ Since
\[
|\tau(x_{n}y)-\tau(x_{0}y)|=|\tau((x_{n}-x_{0})y)|\leq \|x_{n}-x_{0}%
\|_{1}\|y\| \rightarrow0,
\]
we have
\[
|\tau(x_{0}y)|=\lim_{n\rightarrow \infty}|\tau(x_{n}y)|.\] By the   definition of $\overline {\alpha}$, we have that \[|\tau(x_{0}y)|=\lim_{n\rightarrow \infty}|
\tau(x_{n}y)|\leq \limsup_{n\rightarrow
\infty}\overline {\alpha}(x_{n})\alpha^{\prime}(y)\leq M,
\]
whence $\overline {\alpha}(x_{0})\leq M.$ This implies $x_{0}\in L_{\overline {\alpha}}(\mathcal M,\tau)$.
Furthermore, since $\{x_{n}\}$ is Cauchy in $L_{\overline {\alpha}}(\mathcal M,\tau)$, it follows that, for each $n\ge 1$,
\begin{align*}
|\tau((x_{0}-x_{n})y)| & =\lim_{m\rightarrow \infty}|\tau((x_{m}-x_{n})y)|\\
& \leq \limsup_{m\rightarrow \infty}\overline {\alpha}(x_{m}-x_{n})\alpha^{\prime}(y)\\
& \leq \limsup_{m\rightarrow \infty}\overline {\alpha}(x_{m}-x_{n}).
\end{align*}
Thus $\overline {\alpha}(x_{n}-x_{0})\le \limsup_{m\rightarrow \infty}\overline {\alpha}(x_{m}-x_{n})$ for each $n\ge 1$. Again from the fact that $\{x_{n}\}$ is Cauchy in $L_{\overline {\alpha}}(\mathcal M,\tau)$, we conclude that $ \overline {\alpha}(x_n-x_0)\rightarrow 0$ as $n$ goes to infinity.

Therefore $L_{\overline {\alpha}}(\mathcal M,\tau)$ is a  Banach space with respect to the norm $\overline{\alpha}$. This ends the proof of the whole proposition.
\end{proof}

%\subsection{Dual spaces}

\subsection {H\"{o}lder's inequality}

\vspace{0.5cm} In this subsection, we will prove the H\"{o}lder's
inequality for $L^\alpha(\mathcal M,\tau)$ when $\alpha$  is a
normalized, unitarily invariant, $\|\cdot \|_1$-dominating,
continuous norm.

We will need the following result from \cite {Ta}.
\begin{lemma}
\label{completely additive} (Corollary III.3.11 in \cite {Ta}) Let $\mathcal{M}$ be a finite von Neumann algebra
with a faithful normal tracial state $\tau.$ If $\phi$ is a bounded linear
functional on a von Neumann algebra $\mathcal{M},$ then the following two
statements are equivalent:

\begin{enumerate}
\item $\phi$ is normal;

\item For every orthogonal family $\{e_{i}\}_{i\in I}$ in $\mathcal{M},$
\[
\phi(\sum_{i\in I}e_{i})=\sum_{i\in I}\phi(e_{i}).
\]

\end{enumerate}
\end{lemma}

When $\alpha$  is a continuous norm, the following result relates
the dual space of $L^\alpha(\mathcal M,\tau) $ to the space
$L_{\overline {\alpha'}}(\mathcal M,\tau)$.
\begin{proposition}\label{prop3.7}
\label{dual space} Let $\mathcal M$ be a finite von Neumann algebra
with a faithful normal tracial state $\tau.$ Let $\alpha$ be a normalized, unitarily
invariant, $\|\cdot \|_1$-dominating, continuous norm on $\mathcal
M$ (see Definition \ref{def2.2}).    Let   $L_{\overline
{\alpha'}}(\mathcal M,\tau)$ be as {in Definition \ref{def2.10}}.

 Then
for every bounded linear functional $\phi \in(L^{\alpha}(\mathcal{M},\tau))^{\sharp
},$ there is a $\xi \in L_{\overline {\alpha'}}(\mathcal M,\tau)$ such that
$\overline {\alpha'}(\xi)=\| \phi \|$ and $\phi(x)=\tau(x\xi)$ for all $x\in
\mathcal{M}.$
\end{proposition}

\begin{proof}
Suppose $\alpha \in  {N}_{c}(\mathcal M,\tau)$ and $\phi \in(L^{\alpha}(\mathcal{M},\tau%
))^{\sharp}.$ Let $\{e_{n}\}$ be a family of orthogonal projections in
$\mathcal{M}.$ It is easily verified that $\sum_{n=N}^{\infty}e_{n}%
\rightarrow 0$ in the strong operator topology as $N$ approaches infinity. Since $\tau$ is normal, by
Lemma \ref{completely additive}, we have that $\displaystyle \lim_{N\rightarrow \infty}\tau(\sum_{n=N}^{\infty}e_{n})\rightarrow0.$
Note that  $\alpha \in {N}_{c}(\mathcal M,\tau)$. Then the continuity of $\alpha$ with respect to $\tau$ implies that  $\displaystyle \lim_{N\rightarrow \infty}\alpha(\sum_{n=N}^{\infty}%
e_{n})\rightarrow0.$ From the fact that $\phi \in(L^{\alpha}(\mathcal{M},\tau))^{\sharp
},$  we know
\[
\lim_{N\rightarrow \infty} \phi(\sum_{n=1}^{\infty}e_{n}-\sum_{n=1}^{N-1}e_{n})=\lim_{N\rightarrow \infty}\phi(\sum_{n=N}^{\infty
}e_{n})=0.
\]
Now Lemma \ref{completely additive} implies that $\phi$ is a normal functional on
$\mathcal{M}$. Hence $\phi$ is in the predual space of $\mathcal{M},$ i.e.  there is a
$\xi \in L^{1}(\mathcal{M},\tau)$ such that $\phi(x)=\tau(x\xi)$ for all
$x\in \mathcal{M}.$ Furthermore, since $\mathcal{M}$ is dense in $L^{\alpha
}(\mathcal{M},\tau),$ we see
\begin{align*}
 \| \phi \| & =\sup \{|\phi(x)|: x\in{\mathcal{M}}, \alpha(x)\leq1|\} \\
& =\sup \{|\tau(x\xi)|: x\in{\mathcal{M}}, \alpha(x)\leq1\} \\
& =\overline{\alpha'}(\xi),
\end{align*}
which implies that $\xi \in L_{\overline {\alpha'}}(\mathcal M,\tau).$ This ends the proof of the result.
\end{proof}

For a finite von Neumann algebra $\mathcal{M}$ acting on a Hilbert
space $\mathcal{H},$ the set of possibly unbounded, closed and
densely defined operators on $\mathcal{H}$ which are affiliated to
$\mathcal{M},$ forms a topological *-algebra where the topology is
the non-commutative topology of convergence in measure \cite{Nelson}.
We will denote this algebra by $\widetilde{\mathcal{M}};$ it is the
closure of $\mathcal{M}$ in the topology just mentioned. We let
$\widetilde {\mathcal{M}}_{+}$ be the set of positive operators in
$\widetilde{\mathcal{M}}.$  Then the trace $$\tau:\mathcal
M_+\rightarrow [0,\infty)$$ can be extended   to a generalized trace
$$\widetilde \tau: \widetilde {\mathcal{M}}_{+} \rightarrow
[0,\infty].$$ We refer to \cite{Se}, \cite{Y} for more details on
the non-commutative integration theory.

We will summarize some properties of the generalized trace on
$\widetilde {\mathcal M}_+$ as follows.

\begin{lemma}
\label{spectrum projection} Let $\mathcal M$ be a finite von Neumann
algebra with  a faithful normal tracial state $\tau$  acting on a
Hilbert space $\mathcal H$. Let $\widetilde{\mathcal{M}}$ be the set
of closed and densely-defined operators affiliated to $\mathcal{M}$
and $\widetilde {\mathcal{M}}_{+}$ be the set of positive operators
in $\widetilde{\mathcal{M}}.$
If $a\in \widetilde{\mathcal{M}}_{+},$ there is a family $\{e_{\lambda}%
\}_{\lambda>0}$ of projections (spectral resolution of $a$) in $\mathcal{M}$ such that
\begin{enumerate}

\item $e_{\lambda}\rightarrow I$ increasingly;

\item $e_{\lambda}a=ae_{\lambda}\in \mathcal{M}$ for every $0< \lambda<\infty;$

\item $\widetilde \tau(a)=\sup_{\lambda>0}\tau(e_{\lambda}a)$ \ \  ($\ \widetilde \tau(a) $ could be infinity);

\item If $a\in L^{1}(\mathcal{M},\tau),$ then $\|e_{\lambda}a-a\|_{1}\rightarrow0.$

%\item $x\in L^{1}(\mathcal{M},\tau)$ if and only if $\tau(|x|)<\infty$.

\end{enumerate}
Moreover, assume that $x$ is an element in $\widetilde {\mathcal
M}$. Then  $x\in L^{1}(\mathcal{M},\tau)$ if and only if $\widetilde
\tau(|x|)<\infty$.
\end{lemma}
\begin{proof}
The result is well-known. More details could be found in Section 1.1 in \cite{Fa} or in \cite{Y}.
\end{proof}

If no confusion arises, we still use $\tau$ to denote  the
generalized trace $\widetilde \tau$ on $
\widetilde{\mathcal{M}}_{+}.$

A consequence of the preceding lemma is the following result.

\begin{corollary}%\label{cor3.9}
\label{alpha=beta} Let $\mathcal M$ be a finite von Neumann algebra
with a faithful normal tracial state $\tau$  acting on a Hilbert
space $\mathcal H$. Let $\alpha$ be a normalized, unitarily
invariant, $\|\cdot \|_1$-dominating, continuous norm on $\mathcal
M$ (see Definition \ref{def2.2}). Let $\alpha'$ be the dual norm of
$\alpha$ on $\mathcal M$ (see Definition \ref{def2.8}).  Let
$\overline{\alpha}$ and $\overline{\alpha'}$ be as defined {in
Definition \ref{def2.10}}.  Then
$$\alpha(x)=\overline{\alpha}(x)\qquad \text{ and } \qquad \alpha^{\prime}(x)=\overline{\alpha'}(x) \ \ \text{ for all
}  \ x\in \mathcal{M}.$$
\end{corollary}

\begin{proof}
It is clear from Lemma \ref{lemma3.2}   that $\alpha^{\prime}(x)=\overline{\alpha'}(x)$ and $\overline{\alpha}
(x) \leq \alpha(x)$ for all $x\in \mathcal{M}.$ We will need only to show that  $\overline{\alpha}
(x) \ge \alpha(x)$ for all $x\in \mathcal{M}.$

Now
suppose $x\in \mathcal{M}$ with $\alpha(x)=1.$ By the Hahn-Banach theorem,
there is a continuous linear functional $\phi \in(L^{\alpha}(\mathcal{M},\tau%
))^{\sharp}$ such that $\phi(x)=\alpha(x)=1$ and $\| \phi \|=1.$ Since $\phi
\in(L^{\alpha}(\mathcal{M},\tau))^{\sharp},$ from Proposition \ref {prop3.7},  there is an element $\xi
\in L_{\overline {\alpha'}}(\mathcal M,\tau)$ such that $\phi(x)=|\tau
(x\xi)|=1$ and $\overline {\alpha'}(\xi)=\| \phi \|=1.$

Let $\xi=uh$ be the polar decomposition of $\xi \in  L_{\overline
{\alpha'}}(\mathcal M,\tau),$ where $u\in \mathcal{M}$ is a unitary
and $h\in  L_{\overline {\alpha'}}(\mathcal M,\tau)\subseteq
L^{1}(\mathcal{M})$ is positive.   Then it follows from Lemma
\ref{spectrum projection} that  there exists  a family
$\{e_{\lambda}\}_{\lambda>0}$ of projections in  $\mathcal{M}$ such
that \begin{align}\|h-he_{\lambda}\|_{1}\rightarrow0   \tag{$*$}
\end{align}  and  $ e_{\lambda}h =he_{\lambda}\in \mathcal{M} $  for
every $0<\lambda<\infty.$ Thus  $uhe_{\lambda}\in \mathcal{M}.$ We
see that
\begin{align}
\alpha^{\prime}(uhe_{\lambda})=\overline{\alpha'}(uhe_{\lambda})\leq\overline{\alpha'}(uh)\|e_{\lambda}\|
\leq \overline{\alpha'}(uh)=\overline{\alpha'}(\xi)=1. \tag{ by
Lemma \ref{lemma3.2} \& \ref{property of beta2}) ($**$}
\end{align}
Therefore,
\begin{align}
|\tau(x\xi)| =|\tau(xuh)|&=\lim_{\lambda \rightarrow
\infty}|\tau(xuhe_{\lambda })| \tag{by ($*$) and $xu\in\mathcal M$}\\
&\leq \sup \{|\tau(xy)|:y\in{\mathcal{M}},
\alpha^{\prime}(y)\leq1\}. \tag{by ($**$)}
\end{align}
Hence, from the definition of $\overline{\alpha}$ we obtain
\[
\overline{\alpha}(x) =\sup \{|\tau(xy)|: y\in{\mathcal{M}},
\alpha^{\prime}(y)\leq1\} \geq|\tau(x\xi)|=1=\alpha(x).
\]
This finishes the proof of the result.
\end{proof}

A quick corollary of the preceding result is the following conclusion.

\begin{proposition}\label{prop3.10}
 Let $\mathcal M$ be a finite von Neumann algebra with a faithful
normal tracial state $\tau$  acting on a Hilbert space $\mathcal H$.
Let $\alpha$ be a normalized, unitarily invariant, $\|\cdot
\|_1$-dominating, continuous norm on $\mathcal M$ (see Definition
\ref{def2.2}). Let $\alpha'$ be the dual norm of $\alpha$ on
$\mathcal M$ (see Definition \ref{def2.8}).  Let $\overline{\alpha}$
and $\overline{\alpha'}$ be as defined {in Definition
\ref{def2.10}}.

There are natural embeddings
$$
L^\alpha(\mathcal M,\tau) \hookrightarrow
L_{\overline{\alpha}}(\mathcal{M},\tau ) \quad \text{ and } \quad
L^{\alpha'}(\mathcal M,\tau) \hookrightarrow
L_{\overline{\alpha'}}(\mathcal{M},\tau )  \qquad \text{
isometrically},
$$ such that
$$
x \mapsto x \quad \text{ and } \quad  x \mapsto x, \qquad \forall \
x\in \mathcal M.$$ Thus $ L^\alpha(\mathcal M,\tau)$ and
$L^{\alpha'}(\mathcal M,\tau)$ are Banach subspaces of $
L_{\overline{\alpha}}(\mathcal{M},\tau )$, and $
L_{\overline{\alpha'}}(\mathcal{M},\tau )  $ respectively.
\end{proposition}

%From now on, we are ready to embed $L^{\alpha}(\mathcal{M})$ and
%$L^{\alpha^{\prime}}(\mathcal{M})$ into the Banach spaces ${\mathcal{L}%
%}_{\beta}(\mathcal{M})$ and ${\mathcal{L}}_{\beta^{\prime}}(\mathcal{M})$
%respectively. Suppose $x\in L^{1}(\mathcal{M}).$ We can define $\alpha(x)$ and
%$\alpha^{\prime}(x)$ by%

%\[
%\alpha(x)=\beta(x)=\sup \{|\tau(xy)|: y\in{\mathcal{M}}, \alpha^{\prime}%
%(y)\leq1\}
%\]
%and
%\[
%\alpha^{\prime}(x)=\beta^{\prime}(x)=\sup \{|\tau(xy)|: y\in{\mathcal{M}},
%\alpha(y)\leq1\}.
%\]

%\textbf{Throughout the rest of the paper, we write $\beta=\alpha,$
%$\beta^{\prime}=\alpha^{\prime}$ and }

%\textbf{%
%\[
%{\mathcal{L}}_{\beta}(\mathcal{M})={\mathcal{L}}^{\alpha}(\mathcal{M}%
%)=\{x\in{\mathcal{M}}: \alpha(x)<\infty \},
%\]
%\[
%{\mathcal{L}}_{\beta^{\prime}}(\mathcal{M})={\mathcal{L}}^{\alpha^{\prime}%
%}(\mathcal{M})=\{x\in{\mathcal{M}}: \alpha^{\prime}(x)<\infty \}.
%\]
%}

The following theorem is a generalization of H\"{o}lder's inequality
in non-commutative $L^p$-spaces.

\begin{theorem}\label{thm3.11}
Let $\mathcal M$ be a finite von Neumann algebra with a faithful
normal tracial state $\tau$  acting on a Hilbert space $\mathcal H$.
Let $\alpha$ be a normalized, unitarily invariant, $\|\cdot
\|_1$-dominating, continuous  norm on $\mathcal M$ (see Definition
\ref{def2.2}). Let $\alpha'$ be the dual norm of $\alpha$ on
$\mathcal M$ (see Definition \ref{def2.8}).  Let $L_{\overline
{\alpha}}(\mathcal M,\tau)$ and $L_{\overline {\alpha'}}(\mathcal
M,\tau)$ be as defined {in Definition \ref{def2.10}}.

 If $x\in L_{\overline {\alpha}}(\mathcal M,\tau)$ and $y\in{L_{\overline {\alpha'}}(\mathcal M,\tau)}%
 ,$ then $xy\in L^{1}(\mathcal{M},\tau)$ and $\|xy\|_{1}\leq
\overline {\alpha}(x)\overline {\alpha'}(y).$ In particular, if $x\in L^{  {\alpha}}(\mathcal M,\tau)$ and $y\in{L_{\overline {\alpha'}}(\mathcal M,\tau)}%
 ,$ then $xy\in L^{1}(\mathcal{M},\tau)$ and $\|xy\|_{1}\leq
  {\alpha}(x)\overline {\alpha'}(y).$
\end{theorem}

\begin{proof}
Suppose $x\in{L_{\overline{\alpha}}(\mathcal M, \tau)}\subseteq
L^1(\mathcal M,\tau)$ and $y\in{L_{\overline {\alpha'}}(\mathcal
M,\tau)}\subseteq L^1(\mathcal M,\tau).$ Then $xy\in
\widetilde{\mathcal{M}},$ where $\widetilde{\mathcal{M}}$ is the set
of closed and densely defined operators affiliated with
$\mathcal{M}.$ Let $xy=uh$ be the polar decomposition of $xy $ in
$\widetilde{\mathcal{M}}$, where $u\in \mathcal{M}$ is a unitary and
$h=|xy|\in \widetilde{\mathcal{M}}_{+}.$ From   Lemma \ref{spectrum
projection}, there exists
   an increasing  family $\{e_{\lambda}\}_{\lambda>0}$ of projections  in
$\mathcal{M},$   such that $e_{\lambda}h=he_{\lambda}\in \mathcal M$ for each $\lambda>0$ and such that $\tau(h)=\sup
_{\lambda>0}\tau(e_{\lambda}h) $.
We will show that $\tau(h)\le \overline {\alpha}(x)\overline {\alpha'}(y).$

Assume, to the contrary, that  $$\tau(h)=\sup_{\lambda>0}%
\tau(e_{\lambda}h)>\overline {\alpha}(x)\overline {\alpha'}(y).$$ Then there is a projection
$e\in \mathcal{M}$ and $\epsilon>0$ such that   $eh\in \mathcal{M}$ and  $$\tau(eh)> \overline {\alpha}(x)\overline {\alpha'}(y)+\epsilon.$$ Note that $eh=eu^{*}xy.$ We
let $eu^{*}x=h_{2}u_{2}$ be the polar decomposition of $eu^{*}x$ in
$\widetilde{\mathcal{M}},$ where $u_{2}\in \mathcal{M}$ is a unitary and $h_{2}%
\in \widetilde{\mathcal{M}}_{+}.$ Again from  Lemma \ref{spectrum projection},  we may choose $\{f_{\lambda}\}_{\lambda>0}$
to be an increasing family of projections   in $\mathcal{M} $ such that (i) $f_{\lambda}\rightarrow I$ increasingly in the
strong operator topology, (ii) $f_{\lambda}h_2=h_2f_\lambda\in\mathcal M$, and (iii) $\tau(eu^*xu_2^*)=\tau(h_2)=\sup_\lambda \tau(f_{\lambda}h_2)$. From (ii), we have
$f_{\lambda}h_{2}u_{2}\in \mathcal{M}$ for each $\lambda>0$.  It follows that, for each $\lambda>0$,
\begin{align}
|\tau(f_{\lambda}eh)| & =|\tau(f_{\lambda}eu^{*}xy)|=|\tau(f_{\lambda}h_{2}
u_{2} y)|\notag\\
& \leq \alpha(f_{\lambda}h_{2} u_{2})\overline\alpha^{\prime}(y)\tag{by definition of $\overline{\alpha'}$}\\
&=\overline{\alpha}(f_{\lambda}h_{2} u_{2})\overline\alpha^{\prime}(y)\tag{by Corollary \ref{alpha=beta}}\\
& \leq \|f_{\lambda}\| \overline{\alpha}( h_{2}u_{2})\overline{\alpha'}(y)\tag{by Lemma \ref{property of beta2}}\\
& \le \overline{\alpha}(h_{2})\overline{\alpha^{\prime}}(y)\tag{by properties  of $\overline{\alpha}$}\\
& = \overline{\alpha}(eu^{*}xu_{2}^{*})\overline{\alpha'}(y)\notag\\
& \leq \|e\| \overline{\alpha}(u^{*}xu_{2}^{*})\overline{\alpha'}(y)\tag{by Lemma \ref{property of beta2}}\\
& \le \overline{\alpha}(x)\overline{\alpha'}(y). \tag{by properties  of $\overline{\alpha}$}
\end{align}
Moreover, since $f_{\lambda}\rightarrow I$ increasingly in the
strong operator topology and  $eh\in \mathcal{M},$ we have $f_{\lambda
}eh\rightarrow eh$ in the strong operator topology. Since $\tau$ is normal, $\tau$ is  continuous on the bounded subset of $\mathcal M$ in strong operator topology. Therefore, we have
\[
\tau(eh) =|\tau(eh)| =\lim_\lambda|\tau(f_{\lambda}eh )|\leq \overline{\alpha}(x)\overline{\alpha'}(y),
\]
which is a contradiction. Therefore
\[
\|xy\|_{1}=\tau(|xy|)=\tau(h)\leq \overline\alpha(x)\overline{\alpha'}(y),
\] and $xy\in L^{1}(\mathcal{M})$.

If $x\in L^{  {\alpha}}(\mathcal M,\tau)$ and $y\in{L_{\overline {\alpha'}}(\mathcal M,\tau)}$, then, from Proposition \ref{prop3.10}, $\alpha(x)=\overline{\alpha}(x).$ Hence, $
\|xy\|_{1} \leq \alpha(x)\overline{\alpha^{\prime}}(y).
$
\end{proof}

\subsection{Dual space of $L^{\alpha}(\mathcal{M},\tau)$}

Now we are ready to describe the dual space of
$L^{\alpha}(\mathcal{M},\tau)$, when $\alpha$ is a normalized,
unitarily invariant, $\|\cdot\|_1$-dominating and continuous norm on
$\mathcal M$.

\begin{theorem} \label{dualSpace}
Let $\mathcal M$ be a finite von Neumann algebra with a faithful
normal tracial state $\tau$. Let $\alpha$ be a normalized, unitarily
invariant, $\|\cdot \|_1$-dominating, continuous norm on $\mathcal
M$ (see Definition \ref{def2.2}).   Let  $L_{\overline
{\alpha'}}(\mathcal M,\tau)$ be as defined {in Definition
\ref{def2.10}}.

 Then $(L^{\alpha
}(\mathcal{M},\tau))^{\sharp}={L_{\overline {\alpha'}}(\mathcal M,\tau)}, $ i.e.,
\begin{enumerate}
\item [(i)] for every $\phi \in(L^{\alpha}(\mathcal{M},\tau))^{\sharp},$ there is a $\xi
\in{L_{\overline {\alpha'}}(\mathcal M,\tau)}$ such that $\overline{\alpha^{\prime
}}(\xi)=\| \phi \|$ and $\phi(x)=\tau(x\xi)$ for all $x\in {L}^{\alpha
}(\mathcal{M},\tau).$
\item [(ii)] for every $\xi\in {L_{\overline {\alpha'}}(\mathcal M,\tau)}$, the mapping $\phi: L^\alpha(\mathcal M, \tau)\rightarrow \mathbb C$, defined by $\phi(x)=\tau(x\xi)$ for all $x$ in $L^{\alpha
}(\mathcal{M},\tau)$, is in  $(L^{\alpha
}(\mathcal{M},\tau))^{\sharp}$. Moreover, $\|\phi\|=\overline {\alpha'}(\xi).$

\end{enumerate}
\end{theorem}

\begin{proof} (i) Assume that $\phi \in(L^{\alpha}(\mathcal{M},\tau))^{\sharp}.$
From Proposition \ref{prop3.7}, there exists a $\xi \in L_{\overline {\alpha'}}(\mathcal M,\tau)$ such that
$\overline {\alpha'}(\xi)=\| \phi \|$ and $\phi(y)=\tau(y\xi)$ for all $y\in
\mathcal{M}.$ Thus we need only to show that $\phi(x)=\tau(x\xi)$ for all $x\in {L}^{\alpha
}(\mathcal{M},\tau).$

Suppose $x\in L^{\alpha}(\mathcal{M},\tau).$ Then there is a sequence $\{x_{n}\}$
in $\mathcal{M}$ such that $\alpha(x_{n}-x)\rightarrow0.$ Note that $\phi
\in(L^{\alpha}(\mathcal{M},\tau))^{\sharp}.$ Then $\phi(x_{n}-x)\rightarrow0.$ By  the generalized H\"{o}lder's inequality (Theorem \ref{thm3.11}), we have
\[
|\tau(x_{n}\xi)-\tau(x\xi)|=|\tau((x_{n}-x)\xi)|\leq \alpha(x_{n}%
-x)\overline{\alpha'}(\xi)\rightarrow0.
\]
Thus $\tau(x\xi)=\lim_{n\rightarrow \infty}\tau(x_{n}\xi)=\lim_{n\rightarrow
\infty}\phi(x_{n})=\phi(x).$

(ii) It follows directly from  the definition $\overline{\alpha'}$ in Definition \ref{def2.10} and the fact that $\mathcal M$ is dense in $L^{\alpha}(\mathcal{M},\tau)$ that  $$\|\phi\|=\sup\{|\phi(x)|: x\in \mathcal M, \alpha(x)\leq 1\}=\sup\{|\tau(x\xi)|: x\in \mathcal M, \alpha(x)\leq 1\}=\overline{\alpha^\prime}(\xi)<\infty,$$ and thus $\phi\in (L^{\alpha
}(\mathcal{M},\tau))^{\sharp} .$
\end{proof}

\section{Non-commutative Hardy spaces $H^{\alpha}$}

Let $\mathcal{M}$ be a finite von Neumann algebra with a faithful
normal tracial state $\tau.$ Given a von Neumann subalgebra
$\mathcal D$ of $\mathcal M$,  a conditional expectation $\Phi:
\mathcal M\rightarrow \mathcal D$ is defined to be a positive linear
map which preserves the identity and satisfies
$\Phi(x_1yx_2)=x_1\Phi(y)x_2$ for all $x_1, x_2\in \mathcal D$ and
$y\in \mathcal M.$
 For a finite von Neumann algebra  $\mathcal{M}$ with a faithful
normal tracial state $\tau$ and a von Neumann subalgebra $\mathcal
D$, it is a well-known fact that there exists a unique, faithful,
normal, conditional expectation $\Phi$ from  $\mathcal M$ onto
 $\mathcal D$ such that $\tau(\Phi (y))=\tau(y),$ for all $y\in
\mathcal M.$ Furthermore it is known that  such $\Phi:
\mathcal M\rightarrow \mathcal D$ can be extended to a contractive  linear mapping $  {\Phi}:L^1(\mathcal M, \tau)\rightarrow L^1(\mathcal D,\tau)$ satisfying $\tau(y)=\tau( \Phi(y))$ for all $y\in L^1(\mathcal M, \tau)$ (for example, see Proposition 3.9 in \cite{MG}.)

\subsection{Arveson's non-commutative Hardy spaces}

We now recall  non-commutative analogue of classical Hardy space
$H^\infty(\mathbb T)$ by Arveson  in \cite{Arv} (also see
\cite{Exel}).

\begin{definition} Suppose $\mathcal M$ is a finite von Neumann algebra with a faithful
normal tracial state $\tau$. Let  $\mathcal A$ be a weak* closed
unital subalgebra of $\mathcal M,$ and let $\Phi$ be a faithful,
normal conditional  expectation from $\mathcal M$ onto the diagonal
von Neumann algebra $\mathcal D=\mathcal A\cap \mathcal A^*.$ Then
$\mathcal A$ is called a finite, maximal subdiagonal subalgebra of
$\mathcal M$ with respect to $\Phi$ if

\begin{enumerate}
\item $\mathcal A+\mathcal A^*$ is weak* dense in $\mathcal M,$
\item $\Phi(xy)=\Phi(x)\Phi(y)$ for all $x, y\in \mathcal A,$
\item $\tau\circ \Phi=\tau.$
\end{enumerate}

Such a finite, maximal subdiagonal subalgebra $\mathcal A$ of
$\mathcal M$ is also called an $H^\infty$ space of $\mathcal M$.
\end{definition}

\begin{example} Let $\mathcal M=M_n(\mathbb C)$ be the algebra of $n\times n$ matrices with complex entries equipped with a trace $\tau.$ Let $\mathcal A$ be the subalgebra of upper triangular matrices. Now $\mathcal D$ is the diagonal matrices and $\Phi$ is the natural projection onto the diagonal. Then $\mathcal A$ is a finite maximal subdiagonal algebra of $\mathcal M$.
\end{example}

\begin{example} Let $\mathcal M= L^\infty(X,\mu)$, where $(X,\mu)$ is a  probability space. Let $\tau(f)=\int f d\mu $ for all $f$
  in  $L^\infty(X,\mu)$.
 Let $\mathcal A$ be a weak* closed subalgebra of $L^\infty(X,\mu)$ such that $I\in \mathcal A,$ $\mathcal A+\mathcal A^*$
 is weak* dense in $L^\infty(X,\mu),$ and such that $\int fgd\mu=(\int f d\mu)(\int gd\mu)$ for all $f, g\in \mathcal A.$
  Let $\Phi(f)=(\int fd\mu)I$  for all $f$
  in  $L^\infty(X,\mu)$.
   Then $\mathcal A$ is a finite, maximal subdiagonal algebra in $ L^\infty(X,\mu).$
  These examples are the weak* Dirichlet algebras of Srinivasan and Wang \cite{SW}.

\end{example}

\subsection{Non-commutative $H^\alpha$-spaces}
Let $H^\infty$ be a  finite, maximal subdiagonal subalgebra
  of $\mathcal M$. We let
$$H_0^\infty=\{x\in H^\infty: \Phi(x)=0\}.$$

 For $\mathcal S\subseteq
L^p(\mathcal M,\tau),  0< p<\infty,$ let $[\mathcal S]_p$ denote the
closure of $\mathcal S$ in $L^p(\mathcal M,\tau)$ with respect to $\|\cdot \|_p$. Let  $$\text{$H^p=[H^\infty]_p$ \qquad  and \qquad
$H_0^{p}=[H_0^\infty]_p.$}$$ For $\mathcal S\subseteq
 \mathcal M ,$ let $\overline{\mathcal S}^{w*}$ denote the
weak*-closure of $\mathcal S$ in $ \mathcal M $.

The following characterization of non-commutative $H^{p}$ space for
$1\leq p\leq \infty$  was proved by  Saito in \cite{Sai}.

\begin{proposition}{\em (from \cite{Sai})}
\label{property of Hp}Let $1\leq p\leq \infty.$

\begin{enumerate}
\item $H^{1}\cap L^{p}(\mathcal{M},\tau)=H^{p}$ and $H_{0}^{1}\cap L^{p}%
(\mathcal{M},\tau)=H_{0}^{p}.$

\item $H^{p}=\{x\in L^{p}(\mathcal{M},\tau):\tau(xy)=0 \  \mbox{for all } y\in
H_{0}^{\infty}\}.$

\item $H_{0}^{p}=\{x\in L^{p}(\mathcal{M},\tau):\tau(xy)=0\  \mbox{for all } y\in
H^{\infty}\}=\{x\in H^{p}: \Phi(x)=0\};$

\end{enumerate}
\end{proposition}

Similarly, we have the following definition in $L^\alpha(\mathcal M,\tau)$ spaces.
\begin{definition}Suppose $\mathcal M$ is a finite von Neumann algebra with a faithful
normal tracial state $\tau$. Let $H^\infty$ be a  finite, maximal
subdiagonal subalgebra
  of $\mathcal M$.  Suppose $\alpha$ is a  normalized,
unitarily invariant, continuous, $\| \cdot\|_{1}$-dominating norm on
$\mathcal{M}.$

 For $\mathcal S\subseteq
L^\alpha(\mathcal M,\tau),$ let $[\mathcal S]_\alpha$ denote the
closure of $\mathcal S$ in $L^\alpha(\mathcal M,\tau)$ with respect
to the norm $\alpha$. In particular,  We define $H^{\alpha}$ to be
the $\alpha$-closure of $H^{\infty},$ i.e.,
\[
H^{\alpha}=[H^{\infty}]_{\alpha}.
\]
\end{definition}

\subsection{Characterizations of $H^\alpha$-spaces}
In this section, our object is to provide an analogue of Saito's
result stated in Proposition \ref{property of Hp} in the new setting
$H^{\alpha},$ where $\alpha$ is a normalized, unitarily invariant,
 $\| \cdot \|_{1}$-dominating, continuous norm on $\mathcal{M}.$

It is proved in  \cite{BX}, surprisingly, that the multiplication of the
conditional expectation $\Phi$ on $H^{\infty}$ extends to a
multiplication on $H^{p}$ for all $0<p<\infty.$

\begin{lemma}{\em (from \cite{BX})}
\label{multiplicative on Hardy space} The conditional expectation $\Phi$ is
multiplicative on Hardy spaces. More precisely, $\Phi(ab)=\Phi(a)\Phi(b)$ for all
$a\in H^{p}$ and $b\in H^{q}$ with $0<p, q\leq \infty.$
\end{lemma}

Next we will prove two lemmas before we state the main result of the
section.

\begin{lemma}
\label{charcterization1} Let $\mathcal M$ be a finite von Neumann
algebra with a faithful normal tracial state $\tau$, and  $H^\infty$
be a  finite, maximal subdiagonal subalgebra
  of $\mathcal M$.   Let
$\alpha$ be a normalized, unitarily invariant, $\|\cdot
\|_1$-dominating, continuous norm on $\mathcal M$ (see Definition
\ref{def2.2}).   Let $L_{\overline {\alpha'}}(\mathcal M,\tau)$ be
as defined {in Definition \ref{def2.10}}.

Then
\[
H^{\alpha}=\{x\in L^{\alpha}(\mathcal{M},\tau): \tau(xy)=0 \  \mbox{for all } y\in
H_{0}^{1}\cap L_{\overline {\alpha'}}(\mathcal M,\tau) \}.
\]

\end{lemma}

\begin{proof}
Let $$X=\{x\in L^{\alpha}(\mathcal{M},\tau): \tau(xy)=0 \  \mbox{for
all } y\in H_{0}^{1} \cap L_{\overline {\alpha'}}(\mathcal
M,\tau)\}.$$

 Suppose $x\in H^{\infty}.$ If $y\in H_{0}^{1}\cap
L_{\overline {\alpha'}}(\mathcal M,\tau)\subseteq H_{0}^{1},$ then
it follows from part (3) of Proposition \ref{property of Hp} that
$\tau(xy)=0,$ which implies $x\in X,$ and so $H^{\infty }\subseteq
X.$

We claim  that $X$ is $\alpha$-closed in
$L^{\alpha}(\mathcal{M},\tau).$ In fact, suppose  $\{x_{n}\}$ is  a sequence in $X$ and $x\in
L^{\alpha}(\mathcal{M},\tau)$ such that $\alpha(x_{n}-x)\rightarrow 0.$ If $y\in
H_{0}^{1}\cap L_{\overline {\alpha'}}(\mathcal M,\tau),$ then by the
generalized  H\"{o}lder's inequality (Theorem \ref{thm3.11}), we
have
\[
|\tau(xy)-\tau(x_{n}y)|=|\tau((x-x_{n})y)|\leq \alpha(x-x_{n})\overline{\alpha^\prime}(y)\rightarrow0.
\]
Since $x_n\in X$ for all $n\in \mathbb N,$ it follows that
$\tau(xy)=\lim_{n\rightarrow \infty}\tau(x_{n}y)=0 $ for all  $y\in
H_{0}^{1}\cap L_{\overline {\alpha'}}(\mathcal M,\tau) $. By the
definition of $X$, we know that $x\in X.$ Hence $X$ is closed in
$L^{\alpha}(\mathcal{M},\tau)$. Therefore
\[
H^{\alpha}=[{H^{\infty}}]_{\alpha}\subseteq X.
\]

Next, we show that $H^{\alpha}=  X$. Assume, via contradiction,
that $H^{\alpha}\subsetneqq X\subseteq
L^{\alpha}(\mathcal{M},\tau).$ By the Hahn-Banach Theorem, there is
a $\phi \in(L^{\alpha}(\mathcal M,\tau))^{\sharp}$ and $x\in X$ such
that
\begin{enumerate}
\item [(i)] $\phi(x)\neq0,$ and
\item [(ii)] $\phi(y)=0$ for all $y\in H^{\alpha}.$
\end{enumerate}
Since $\alpha$ is a normalized, unitarily invariant, $\|\cdot
\|_1$-dominating, continuous norm on $\mathcal M$, it follows from
Proposition \ref{dual space} that there exists a $\xi \in
L_{\overline {\alpha'}}(\mathcal M,\tau)$ such that
\begin{enumerate}
\item [(iii)]
$\phi (z)=\tau(z\xi)$ for all $z\in
L^{\alpha}(\mathcal{M},\tau).$
\end{enumerate}
Hence from (ii) and (iii)  we can conclude that
\begin{enumerate}
\item [(iv)] $\tau(y\xi)=\phi(y)=0 $ for every $y\in H^\infty\subseteq H^\alpha\subseteq
L^\alpha(\mathcal M, \tau).$ \end{enumerate}

 Since $\xi\in
L_{\overline {\alpha'}}(\mathcal M,\tau)\subseteq L^1(\mathcal M,
\tau),$ it follows from part $(3)$ of
Proposition \ref{property of Hp} and  (iv) as above that $\xi \in H_{0}^{1},$ which means $\xi \in H_{0}%
^{1}\cap L_{\overline {\alpha'}}(\mathcal M,\tau).$ Combining with
the fact that $x\in X=\{x\in L^{\alpha}(\mathcal{M},\tau):
\tau(xy)=0 \ \mbox{for all } y\in H_{0}^{1} \cap L_{\overline
{\alpha'}}(\mathcal M,\tau)\},$ we obtain that $\tau(x\xi)=0.$ Note, again,
that $x\in X\subseteq L^\alpha(\mathcal M,\tau).$ From (i) and
(iii), it follows that $\tau(x\xi)=\phi(x)\neq0.$ This is a
contradiction. Therefore
\[
H^{\alpha}=X=\{x\in L^{\alpha}(\mathcal{M},\tau): \tau(xy)=0 \  \mbox{for all } y\in
H_{0}^{1}\cap L_{\overline {\alpha'}}(\mathcal M,\tau) \}.
\]

\end{proof}

\begin{lemma}
\label{characterization2} Let $\mathcal M$ be a finite von Neumann
algebra with a faithful normal tracial state $\tau$, and  $H^\infty$
be a  finite, maximal subdiagonal subalgebra
  of $\mathcal M$.  Let $\alpha$ be a normalized, unitarily invariant, $\|\cdot \|_1$-dominating, continuous norm on $\mathcal M$ (see
Definition \ref{def2.2}).   Let  $L_{\overline {\alpha'}}(\mathcal
M,\tau)$ be as defined {in Definition \ref{def2.10}}.  Then
\[
H^{1}\cap L^{\alpha}(\mathcal{M},\tau)=\{x\in L^{\alpha}(\mathcal{M},\tau): \tau(xy)=0
\  \mbox{for all } y\in H_{0}^{1}\cap L_{\overline {\alpha'}}(\mathcal M,\tau) \}.
\]

\end{lemma}

\begin{proof}
Let $$X=\{x\in L^{\alpha}(\mathcal{M},\tau): \tau(xy)=0 \  \mbox{for
all } y\in H_{0}^{1} \cap L_{\overline {\alpha'}}(\mathcal
M,\tau)\}.$$ It is clear that $X\subseteq L^{\alpha}(\mathcal{M},\tau).$

Now we suppose $  x\in X$, that is  $x\in
L^{\alpha}(\mathcal{M},\tau)$ such that $\tau(xy)=0 $ for all $y\in
H_{0}^{1} \cap L_{\overline {\alpha'}}(\mathcal M,\tau) .$ Since
$H_{0}^{\infty}\subseteq H^{\infty}\subseteq{\mathcal{M}}\subseteq
L_{\overline {\alpha'}}(\mathcal M,\tau) $ and
$H_{0}^{\infty}\subseteq H_{0}^{1},$ it follows that $\tau(xy)=0$
for all $y\in H_{0}^{\infty}.$ Then by part (2) of Proposition
\ref{property of Hp}, $x\in H^{1},$ which implies $X\subseteq
H^1\cap L^\alpha(\mathcal M, \tau).$

To prove $ H^1\cap L^\alpha(\mathcal M, \tau)\subseteq X,$ suppose
$x\in H^{1}\cap L^{\alpha }(\mathcal{M},\tau ).$  Then $x\in
L^\alpha(\mathcal M, \tau).$ Assume that  $y\in H_{0}^{1}\cap
L_{\overline {\alpha'}}(\mathcal M,\tau).$ So $\Phi(y)=0.$ Note that
$xy\in H^1H^1_0\subseteq H^{1/2}$. From Lemma \ref{multiplicative on Hardy space}, we
know that $\Phi(xy)$  is   in $L^{1/2}(\mathcal D,\tau)$ (see Theorem 2.1 in \cite {BX}) and
$\Phi(xy)=\Phi(x)\Phi(y)=0.$ Moreover, since $x \in
L^{\alpha}(\mathcal{M},\tau)$ and $y\in  L_{\overline
{\alpha'}}(\mathcal M,\tau)$, it induces from Theorem \ref {thm3.11}
that  $xy\in  L^{1}(\mathcal{M},\tau),$ whence $\Phi(xy)$ is also in
$L^{1}(\mathcal{M},\tau)$.  Thus $\tau(xy)$ is well defined and
$\tau(xy)=\tau(\Phi(xy))=0.$ By the definition of $X$, we conclude
that $x\in X.$ Therefore $H^{1}\cap
L^{\alpha}(\mathcal{M},\tau)\subseteq X$.

Now we can obtain that
$$H^{1}\cap L^{\alpha}(\mathcal{M},\tau)= \{x\in
L^{\alpha}(\mathcal{M},\tau): \tau(xy)=0 \  \mbox{for all } y\in
H_{0}^{1}\cap L_{\overline {\alpha'}}(\mathcal M,\tau) \}.$$
\end{proof}

The following theorem gives   a   characterization of
$H^\alpha .$

\begin{theorem}
\label{characterization of H^alpha}Let $\mathcal M$ be a finite von
Neumann algebra with a faithful normal tracial state $\tau$, and
$H^\infty$ be a  finite, maximal subdiagonal subalgebra
  of $\mathcal M$. Let $\alpha$ be a normalized, unitarily invariant, $\|\cdot \|_1$-dominating, continuous norm on $\mathcal M$.  Then
\[
H^{\alpha}=H^{1}\cap L^{\alpha}(\mathcal{M},\tau)=\{x\in L^{\alpha}(\mathcal{M},\tau):
\tau(xy)=0 \  \mbox{for all } y\in H_{0}^{\infty} \}.
\]

\end{theorem}

\begin{proof}
The result follows directly from Lemma \ref{charcterization1}, Lemma
\ref{characterization2} and Proposition \ref{property of Hp}.
\end{proof}

\section{Beurling's invariant subspace theorem}

In this section, we extend  the  classical Beurling's theorem  to
Arveson's non-commutative   Hardy spaces associated with unitary
invariant norms. %Before stating our main result, we recall some
%useful definitions.

%\begin{definition}
%Suppose $\mathcal{M}$ is a finite von Neumann algebra with a
%faithful normal tracial state $\tau $, and  $H^\infty$  be a
%finite, maximal subdiagonal subalgebra
%  of $\mathcal M$. Let $W$ be a closed subspace of $L^{p}(\mathcal{M},\tau).$ We
%say that $W$ is a left (resp. right) $H^\infty$-invariant if
%$H^{\infty}W\subseteq W$ (resp. $W H^{\infty}\subseteq W$).
%\end{definition}

%\begin{definition}(\cite{BL2}) {\textbf{(Need to be revised)}}
%\label{type 2 definition} If $K$ is a right $H^{\infty}$-invariant subspace of
%$L^{p}(\mathcal{M},\tau)$ for all $1\leq p\leq \infty,$ defining the right wandering
%subspace of $K$ to be the space $K=K\ominus[K H_{0}^{\infty}]_{p};$ and we say
%that $K$ is type $1$ if $K$ generates $K$ as an $H^{\infty}$-module (that is,
%$K=[KH^{\infty}]$). We say $K$ is type $2$ if $K=(0)$ (that is, $K=[KH_{0}%
%^{\infty}]_{p}$). If $p=\infty,$ we take $[\cdot]_{p}$ to be the weak* closure.
%\end{definition}

\subsection {A factorization result} In \cite{Sai},  Saito proved the following useful
factorization theorem.

\begin{lemma}
\label{decomposition} {\em (from \cite{Sai}) }Suppose $\mathcal{M}$
is a finite von Neumann algebra with a faithful normal tracial state
$\tau $, and  $H^\infty$  be a finite, maximal subdiagonal
subalgebra
  of $\mathcal M$. If $k\in \mathcal{M}$ and $k^{-1}\in
L^{2}(\mathcal{M},\tau),$ then there are unitary operators $u_{1},
u_{2}\in \mathcal{M}$ and operators $a_{1}, a_{2}\in H^{\infty}$
such that $k=u_{1}a_{1}=a_{2}u_{2}$ and $a_{1}^{-1}, a_{2}^{-1}\in
H^{2}.$
\end{lemma}

We shall show that in fact it is possible to choose $a_{1}$ and
$a_{2}$ with their inverses   in $H^{\alpha}.$

\begin{proposition}
\label{decomposition in L^alpha} Let $\mathcal M$ be a finite von
Neumann algebra with a faithful normal tracial state $\tau $, and
$H^\infty$  be a finite, maximal subdiagonal subalgebra
  of $\mathcal M$.  Let
$\alpha$ be a normalized, unitarily invariant, $\|\cdot
\|_1$-dominating, continuous norm on $\mathcal M$. If $k\in
\mathcal{M}$ and $k^{-1}\in L^{\alpha}(\mathcal{M},\tau),$ then
there are unitary operators $w_{1}, w_{2}\in \mathcal{M}$ and
operators $a_{1}, a_{2}\in H^{\infty}$ such that
$k=w_{1}a_{1}=a_{2}w_{2}$ and $a_{1}^{-1}, a_{2}^{-1}\in
H^{\alpha}.$
\end{proposition}

\begin{proof}
Suppose $k\in \mathcal{M} $  with $k^{-1}\in
L^{\alpha}(\mathcal{M},\tau).$ Assume that $k=vh$ is the polar
decomposition of $k$ in   $\mathcal M$, where $v$ is a unitary
operator in $\mathcal M$ and $h$ in $ \mathcal{M}$ is positive. Then
from the assumption that $k^{-1}=h^{-1}v^{*}\in
L^{\alpha}(\mathcal{M},\tau),$ we see $h^{-1}\in
L^{\alpha}(\mathcal{M},\tau)\subseteq L^{1}(\mathcal{M},\tau).$
Since $h$ in $\mathcal M$ is positive,  we can conclude that
 $h^{-\frac{1}{2}}\in L^{2}(\mathcal{M},\tau).$ Note that
$h^{\frac{1}{2}}\in \mathcal{M}.$ It follows from Lemma
\ref{decomposition} that there exist a unitary operator $u_{1}\in
\mathcal{M}$ and $h_{1}\in H^{\infty}$ such that
$h^{\frac{1}{2}}=u_{1}h_{1}$ and $h_{1}^{-1}\in H^{2}.$

Now $h=h^{\frac{1}{2}}\cdot h^{\frac{1}{2}}%
=u_{1}(h_{1}u_{1})h_{1}.$  Since $h_{1}u_{1}$ is in $\mathcal M$ and
$(h_{1}u_{1})^{-1}=u_{1}^{*}h_{1}^{-1}\in L^{2}(\mathcal{M},\tau),$
by Lemma \ref{decomposition}  there exist  a  unitary operator
$u_{2}\in \mathcal{M}$ and $h_{2}\in H^{\infty}$ such that
$h_{1}u_{1}=u_{2}h_{2}$  and $h_{2}^{-1}\in H^{2}.$ Thus
\[
k=vh=vu_{1}h_{1}u_{1}h_{1}=vu_{1}u_{2}h_{2}h_{1}=w_{1}a_{1},
\]
where $w_{1}=vu_{1}u_{2}$ is a  unitary operator in $\mathcal{M}$ and $a_{1}=h_{2}%
h_{1}\in H^{\infty}$ with $$a_{1}^{-1}=(h_{2}h_{1})^{-1}=h_{1}^{-1}h_{2}%
^{-1}\in H^{2}\cdot H^{2}\subseteq H^{1}.$$ Since $k^{-1}=(w_{1}a_{1})^{-1}%
=a_{1}^{-1}w_{1}^{*}\in L^{\alpha}(\mathcal{M},\tau),$ we obtain
that  $a_{1}^{-1}\in L^{\alpha }(\mathcal{M},\tau).$ Then by Theorem
\ref{characterization of H^alpha}, we have
$$a_{1}^{-1}\in H^{1}\cap L^{\alpha}(\mathcal{M})=H^{\alpha}.$$
Hence $w_1$ is a unitary in $\mathcal M$ and $a_1$ is in $H^\infty$ such that $k=w_1a_1$ and $a_1^{-1}\in H^\alpha$.

Similarly, there exist a unitary operator  $w_{2}\in \mathcal{M}$
and $a_{2}\in H^{\infty}$ such that  $k=a_{2}w_{2} $   and
$a_{2}^{-1}\in H^{\alpha}.$
\end{proof}

\subsection{Dense subspaces}

The following theorem plays an important role in the proof of our main result of the paper.

\begin{theorem}\label{invariant on L^alpha} Let $\mathcal M$ be a finite von Neumann algebra with a faithful
normal tracial state $\tau$, and $H^\infty$  be a finite, maximal
subdiagonal subalgebra
  of $\mathcal M$. Let $\alpha$ be a normalized, unitarily invariant, $\|\cdot \|_1$-dominating, continuous norm on $\mathcal M$. % (see
%Definition \ref{def2.2}).  Let  $L_{\overline {\alpha'}}(\mathcal
%M,\tau)$ be as defined {in Definition \ref{def2.10}}.

If $\mathcal W$ is a closed subspace of $L^\alpha(\mathcal M,\tau) $
and $\mathcal N$ is a weak* closed linear subspace of $\mathcal M$
such that $\mathcal  W H^\infty\subseteq \mathcal W$ and $  \mathcal
N H^\infty\subseteq \mathcal N,$  then

\begin{enumerate}
\item $\mathcal N=[\mathcal N]_{\alpha}\cap \mathcal M;$

\item $\mathcal W\cap\mathcal M $ is weak* closed in $\mathcal M;$

\item $\mathcal W=[  \mathcal W\cap\mathcal M]_{\alpha};$

\item if $\mathcal S$ is a subspace of $\mathcal M$ such that $\mathcal S H^\infty\subseteq \mathcal S$, then $$[\mathcal
S]_\alpha= [\overline {\mathcal S}^{w*}]_\alpha,$$ where $\overline
{\mathcal S}^{w*}$ is the weak*-closure of $\mathcal S$ in $\mathcal
M$.
\end{enumerate}\end{theorem}

\begin{proof} (1). It is clear that $\mathcal N\subseteq [\mathcal N]_{\alpha}\cap \mathcal
M.$ Assume, via contradiction, that $\mathcal N\subsetneqq [\mathcal
N]_{\alpha}\cap \mathcal M $. %$\subseteq L^\alpha(\mathcal M,\tau).$
Note that $\mathcal N$ is a weak* closed linear subspace of
$\mathcal M$ and $L^1(\mathcal M,\tau)$ is the predual space of
$\mathcal M$. It follows from the Hahn-Banach Theorem that there
exist  a $\xi\in L^1(\mathcal M, \tau)$   and an $x\in [\mathcal
N]_{\alpha}\cap \mathcal M $ such that \begin{enumerate}
\item [(a)]
$\tau( \xi x)\neq 0,$ but
\item [(b)] $\tau( \xi y)=0$ for all $y\in \mathcal N.$
\end{enumerate}%From Theorem \ref{dual
%space}, there is a $\xi\in (L^\alpha(\mathcal M,
%\tau))^\sharp=L_{\overline {\alpha'}}(\mathcal M,\tau)$ such that
%$\tau( x\xi )=\phi(x)\neq 0$ and $\tau(y\xi)=\phi(y)=0$ for all
%$y\in \mathcal N.$
 We claim that there exists a  $z\in \mathcal M$ such that
\begin{enumerate}
\item [(a')] $\tau(zx )\neq 0,$ but \item [(b')] $\tau(zy )=0$ for all $y\in \mathcal N.$
\end{enumerate}
Actually assume that $\xi= |\xi^*|v$ is the polar decomposition of
$\xi$ in $L^1(\mathcal M,\tau),$ where $v$ is a unitary element in
$\mathcal M $ and $|\xi^*|$ in $L^1(\mathcal M,\tau)$ is positive.
Let $f$ be a function on $[0,\infty)$ defined by the formula
$f(t)=1 $ for $0\leq t\leq 1 $ and $f(t)=1/t$ for $t>1.$ We define
$k=f(|\xi^*|)$ by the functional calculus. Then by the construction
of $f,$ we know that $k\in \mathcal M$ and
$k^{-1}=f^{-1}(|\xi^*|)\in L^1(\mathcal M,\tau).$ It follows from
Theorem \ref{decomposition in L^alpha} that there exist a unitary
$u\in \mathcal M$ and $a\in H^\infty$ such that $k=ua$ and
$a^{-1}\in H^1.$ Therefore, we can further assume that
$\{a_n\}_{n=1}^\infty$ is a sequence of elements in $H^\infty$ such
that $\|a^{-1}-a_n\|_1\rightarrow 0$. Observe that
\begin{enumerate}
\item [(i)] since $a, a_n$ are in $H^\infty$,   for each
$y\in \mathcal N$ we have that $ya_na \in \mathcal
NH^\infty\subseteq \mathcal N$ and
$$
 \tau ( a_n a\xi y)=\tau(  \xi ya_na)=0;
$$
\item [(ii)] we have $a \xi  =  (u^*u) a   (|\xi^*|v) = u^*(k  |\xi^*|)v \in\mathcal
M$, by the definition of $k$;
\item [(iii)] from (a) and (ii), we have
$$
0\ne \tau( \xi x)=\tau( a^{-1}a \xi x   ) =\lim_{n\rightarrow
\infty } \tau( a_n a\xi x   ).
$$
\end{enumerate} Combining (i), (ii) and (iii), we are able to find an $N\in \mathbb N$ such that $z=a_Na\xi\in \mathcal M$ satisfying
\begin{enumerate}
\item [(a')] $\tau(zx )\neq 0,$ but \item [(b')] $\tau(zy )=0$ for all $y\in \mathcal N.$
\end{enumerate}

 Recall that $x\in [\mathcal N]_{\alpha}.$ Then there is a sequence
$\{x_n\}$ in $\mathcal N$ such that $\alpha(x-x_n)\rightarrow 0.$ We
have
$$|\tau(zx_n )-\tau(zx )|=|\tau(z(x-x_n) )|\leq \|x-x_n\|_1\|z\|\leq
\alpha(x-x_n)\|z\|\rightarrow 0.$$ Combining with (b') we conclude
that $\tau(zx )=\lim_{n\rightarrow\infty}\tau(zx_n )=0.$ This
contradicts with the result (a'). Therefore $\mathcal N=[\mathcal
N]_{\alpha}\cap \mathcal M.$

(2).   Let $\overline{\mathcal W\cap \mathcal{M}}^{w*}$ be the weak*-closure of $\mathcal W\cap \mathcal{M}$ in $\mathcal M$. In order to
show that $\mathcal W\cap \mathcal{M}= \overline{\mathcal W\cap \mathcal{M}}^{w*}$, it suffices to show that
 $\overline{\mathcal W\cap \mathcal{M}}^{w*}\subseteq \mathcal W$. Assume, to the contrary,
  that $\overline{\mathcal W\cap \mathcal{M}}^{w*}  \nsubseteq\mathcal W$. Thus there exists an element $x$ in $\overline{\mathcal W\cap \mathcal{M}}^{w*}\subseteq \mathcal M\subseteq L^{\alpha}(\mathcal M,\tau)$, but $x\notin \mathcal W$. Since $\mathcal W$ is a closed subspace of $L^{\alpha}(\mathcal M,\tau)$, by the Hahn-Banach Theorem and Theorem \ref{dualSpace}, there exists a $\xi\in L_{\overline {\alpha'}}(\mathcal M,\tau)\subseteq L^1(\mathcal M,\tau)$  such that $\tau(
 \xi x )\neq 0$ and $\tau( \xi y)=0$ for all $y\in  \mathcal W.$ Since $\xi\in L^1(\mathcal M,\tau)$, the linear mapping $\tau_\xi:\mathcal M\rightarrow \mathbb C$, defined by $\tau_\xi(a)=\tau(\xi a) $ for all $a\in \mathcal M$, is weak*-continuous. Note that $x\in  \overline{\mathcal W\cap \mathcal{M}}^{w*}$ and  $\tau( \xi y)=0$ for all $y\in  \mathcal W.$ We know that $\tau(\xi x)=0$, which   contradicts with the assumption that $\tau(\xi x)\ne 0$. Hence  $\overline{\mathcal W\cap \mathcal{M}}^{w*}\subseteq \mathcal W$, whence $\overline{\mathcal W\cap \mathcal{M}}^{w*}= \mathcal W\cap \mathcal M$.

(3). Since $\mathcal W$ is $\alpha$-closed, it is easy to see
$[\mathcal W\cap {\mathcal M}]_\alpha\subseteq \mathcal W.$ Now we
assume $[\mathcal W\cap {\mathcal M}]_\alpha\subsetneqq \mathcal
W\subseteq L^\alpha(\mathcal M,\tau).$ By the Hahn-Banach Theorem
and Theorem \ref{dualSpace} there exist an $x\in \mathcal W$ and
$\xi\in L_{\overline {\alpha'}}(\mathcal M,\tau)$  such that $\tau(
 \xi x )\neq 0$ and $\tau( \xi y)=0$ for all $y\in [\mathcal W\cap
{\mathcal M}]_\alpha.$ Let $x=v|x|$ be the polar decomposition of $x
$ in $L^\alpha(\mathcal M,\tau)$, where $v$ is a unitary element in
$\mathcal M. $ Let $f$ be a function on $[0,\infty)$ defined by
the formula $f(t)=1$ for $ 0\leq t\leq 1 $  and $f(t)=1/t$ for $
t>1.$ We define $k=f(|x|)$ through the functional calculus. Then we
see $k\in \mathcal M$ and $k^{-1}=f^{-1}(|x|)\in L^\alpha(\mathcal
M,\tau).$ It follows from Theorem \ref{decomposition in L^alpha}
that there exist a unitary $u\in \mathcal M$ and $a\in H^\infty$ such
that $k=au$ and $a^{-1}\in H^\alpha.$ A little computation shows
that $ |x|k\in \mathcal M,$ which implies that
$$xa= x auu^*= x ku^*=v(|x| k)u^*\in \mathcal M.$$ Since $a\in H^\infty,$  we know $  xa\in \mathcal  WH^\infty\subseteq
 \mathcal W,$ and thus $ xa\in \mathcal W\cap\mathcal M.$ Furthermore, note that $(\mathcal W\cap {\mathcal M})H^\infty\subseteq
  \mathcal W\cap \mathcal M$. Thus, if $b\in H^\infty,$ we see $xab\in \mathcal W\cap\mathcal M,$  and so $\tau(\xi xab)=0.$ Since
  $H^\infty$ is dense in $H^\alpha $ and $\xi$ is in $L_{\overline {\alpha'}}(\mathcal M,\tau)$,   it follows from Theorem \ref{thm3.11}
   that $\tau(\xi xab )=0$ for all $b\in H^\alpha.$ Since $a^{-1}\in H^\alpha,$ we see $\tau( \xi x )=\tau(\xi x a a^{-1} )=0.$
   This   contradicts with the assumption that $\tau( \xi x )\ne 0$. Therefore $\mathcal W=[\mathcal W\cap {\mathcal M}]_\alpha.$

(4) Assume that $\mathcal S$ is a subspace of $\mathcal M$ such that
$\mathcal S H^\infty\subseteq \mathcal S$ and $\overline {\mathcal
S}^{w*}$ is the weak*-closure of $\mathcal S$ in $\mathcal M$. Then
$[\mathcal S]_\alpha H^\infty\subseteq [\mathcal S]_\alpha$. Note
that $\mathcal S \subseteq [\mathcal S]_\alpha\cap \mathcal M$. From
(2), we know that $[\mathcal S]_\alpha\cap \mathcal M$ is
weak*-closed. Therefore $\overline {\mathcal S}^{w*} \subseteq
[\mathcal S]_\alpha\cap \mathcal M$. Hence $[\overline {\mathcal
S}^{w*}]_\alpha\subseteq [\mathcal S]_\alpha$, whence $[\overline
{\mathcal S}^{w*}]_\alpha= [\mathcal S]_\alpha$

\end{proof}

%This gives us a quick route from the $\left \Vert {}\right \Vert _{2}$-version
%of invariant subspace structure  to the (weak*-closed) $\left \Vert
%}\right \Vert _{\infty}$-version of invariant subspace structure,
%and then to the $\alpha$-version of invariant subspace structure.

\subsection{Main result}
Before we state our main result in this section, we will need the
following definition from \cite{JS}.
\begin{definition}
  Let $\mathcal M$ be a finite von Neumann algebra with a
faithful, tracial, normal state $\tau$. Let $X$ be a weak* closed
subspace of $\mathcal M$. Then $X$ is called an internal column sum
of a family of weak* closed subspaces $\{X_i\}_{i\in \mathcal I}$ of $\mathcal
M$, denoted by
$$
X=\bigoplus^{col}%
_{i\in \mathcal I} X_i
$$ if
\begin{enumerate}
\item $X_j^*X_i=\{0\}$ for all distinct $i,j\in \mathcal I$; and
\item the linear span of $\{X_i\ :  \ i\in \mathcal I\}$ is
weak* dense in $X$, i.e. $X=\overline{ span  \{X_i\ :  \ i\in \mathcal I\} }^{w*}$.
\end{enumerate}
\end{definition}

Similarly, we introduce a concept of internal column sum of
subspaces in $L^\alpha(\mathcal M,\tau)$ as follows.
\begin{definition}\label{def5.5}
  Let $\mathcal M$ be a finite von Neumann algebra with a
faithful, tracial, normal state $\tau$ and $\alpha$ be a normalized,
unitarily invariant, $\|\cdot\|_1$-dominating and  continuous  norm
on $\mathcal M$. Let $X$ be a closed subspace of $L^\alpha(\mathcal
M,\tau)$. Then $X$ is called an internal column sum of a family of
closed subspaces $\{X_i\}_{i\in \mathcal I}$ of $L^\alpha(\mathcal
M,\tau)$, denoted by
$$
X=\bigoplus^{col}%
_{i\in \mathcal I} X_i
$$ if
\begin{enumerate}
\item $X_j^*X_i=\{0\}$ for all distinct $i,j\in \mathcal I$; and
\item the linear span of $\{X_i\ :  \ i\in \mathcal I\}$ is
 dense  in $X$, i.e. $X=[span  \{X_i\ :  \ i\in \mathcal I\}]_\alpha.$
\end{enumerate}
\end{definition}

In \cite{BL2}, David P. Blecher and Louis E. Labuschagne  proved a
version of Beurling's  theorem for  $L^{p}(\mathcal{M},\tau)$  spaces
when  $1\leq p\leq \infty.$

\begin{lemma} {\em (from \cite{BL2})}
\label{invariant on L^p}  Let $\mathcal M$ be a finite von Neumann
algebra with a faithful, tracial, normal state $\tau$, and $
H^{\infty}$ be a maximal subdiagonal subalgebra of $\mathcal{M}$ with $\mathcal D=H^\infty\cap (H^\infty)^*$.
Suppose that $\mathcal K$ is a closed $H^{\infty}$-right-invariant subspace
of $L^p(\mathcal{M},\tau),$ for some $1\leq p\leq \infty.$ (For
$p=\infty$ we assume that $\mathcal K$ is weak* closed.) Then
$\mathcal K$
may be written as a column $L^{p}$-sum $\mathcal K= \mathcal Z \bigoplus^{col}(\bigoplus^{col}%
_{i}u_{i}H^{p})$, where $ \mathcal Z $ is a closed ( indeed weak*
closed if $p=\infty$)
  subspace of $L^{p}(\mathcal{M},\tau)$ such that $ \mathcal Z =[ \mathcal Z H_{0} ^{\infty}]_{p}$, and where $u_{i}$ are partial
isometries in ${\mathcal{M}}\cap \mathcal K$ with $u_{j}^{*}u_{i}=0$
if $i\neq j,$ and
with $u_{i}^{*}u_{i}\in{\mathcal{D}}.$ Moreover, for each $i,$ $u_{i}%
^{*} \mathcal Z =\{0\},$ left multiplication by the $u_{i}u_{i}^{*}$
are contractive projections from $\mathcal K$ onto the summands
$u_{i} H^{p} ,$ and left multiplication by
$I-\sum_{i}u_{i}u_{i}^{*}$ is a contractive projection from
$\mathcal K$ onto $ \mathcal Z .$
\end{lemma}

Now we are ready to prove the main result of the paper, a
generalized version of the classical theorem of Beurling \cite{B} in
a non-commutative $L^{\alpha}(\mathcal{M},\tau)$ space for a
normalized, unitarily invariant, $\|\cdot \|_1$-dominating,
continuous norm $\alpha.$

\begin{theorem}\label{mainthm}
Let $\mathcal M$ be a finite von Neumann algebra with a faithful,
normal, tracial state $\tau$. Let $H^\infty$  be a finite, maximal
subdiagonal subalgebra
  of $\mathcal M$ and $\mathcal D=H^\infty\cap (H^\infty)^*.$
 Let $\alpha$ be a normalized, unitarily
invariant, $\|\cdot \|_1$-dominating, continuous norm on $\mathcal
M$.

 If $\mathcal W$ is a   closed subspace of
$L^{\alpha}(\mathcal{M},\tau),$ then $\mathcal WH^{\infty}\subseteq
\mathcal W$ if and only if $$\mathcal W= \mathcal Z
\bigoplus^{col}(\bigoplus^{col}_{i\in \mathcal I}u_{i}H^{\alpha}),$$
where $ \mathcal Z $ is a closed  subspace of
$L^{\alpha}(\mathcal{M},\tau)$ such that $ \mathcal Z =[ \mathcal Z
H_{0} ^{\infty}]_{\alpha}$, and where $u_{i}$ are partial
isometries in $\mathcal W  \cap  {\mathcal{M}}$ with
$u_{j}^{*}u_{i}=0$
if $i\neq j,$ and with $u_{i}^{*}u_{i}\in{\mathcal{D}}.$ Moreover, for each $i,$ $u_{i}%
^{*} \mathcal Z =\{0\},$ left multiplication by the $u_{i}u_{i}^{*}$
are contractive projections from $\mathcal W$ onto the summands
$u_{i} H^{\alpha},$ and left multiplication by
$I-\sum_{i}u_{i}u_{i}^{*}$ is a contractive projection from
$\mathcal W$ onto $ \mathcal Z .$
\end{theorem}

\begin{proof}
The only if part is obvious. Suppose $\mathcal W$ is a  closed
subspace of $L^{\alpha}(\mathcal{M},\tau)$ such that $\mathcal
WH^{\infty} \subseteq \mathcal W.$ Then it follows from part (2) of
Theorem \ref{invariant on L^alpha} that $\mathcal
W\cap{\mathcal{M}}$ is weak* closed in ${\mathcal{M}}$.  It follows
from Lemma \ref{invariant on L^p}, in the case $p=\infty$, that
$$ \mathcal W\cap{\mathcal{M}} =
 \mathcal Z _1\bigoplus^{col}(\bigoplus^{col}_{i\in \mathcal
I}u_{i}H^\infty),$$ where $ \mathcal Z _1$ is a weak* closed
subspace in $ \mathcal M  $ such that $ \mathcal Z _1=\overline{
\mathcal Z _1H_0^\infty}^{w*}$,  and $u_i$ are partial isometries in
$  \mathcal W\cap \mathcal M$ with $u_{j}^{*}u_{i}=0$ if $i\neq j,$
and with
$u_{i}^{*}u_{i}\in{\mathcal{D}}.$ Moreover, for each $i,$ $u_{i}%
^{*} \mathcal Z _1=\{0\},$ left multiplication by the
$u_{i}u_{i}^{*}$ are contractive projections from $ \mathcal W\cap
{\mathcal M} $ onto the summands $u_{i}H^{\infty},$ and left
multiplication by $I-\sum_{i}u_{i}u_{i}^{*}$ is a contractive
projection from $ \mathcal W\cap {\mathcal M} $ onto $ \mathcal Z
_1.$

Let $  \mathcal Z =[ \mathcal Z _1]_{\alpha}. $ It is not hard to verify that  for each $i,$ $u_{i}%
^{*} \mathcal Z =\{0\}$.  We also claim that
$[u_{i}H^{\infty}]_{\alpha}=
 u_{i}H^{ \alpha}.$ In fact it is obvious that
 $[u_{i}H^{\infty}]_{\alpha}\supseteq
 u_{i}H^{ \alpha}.$ We will need only to show that
 $[u_{i}H^{\infty}]_{\alpha}\subseteq
 u_{i}H^{ \alpha}.$ Let $\{a_n\}\subseteq H^\infty$ and $a\in
 [u_{i}H^{\infty}]_{\alpha}$ be such that $\alpha(u_ia_n-a)\rightarrow
 0$. By the choice of $u_i$, we know that $u_i^*u_i\in\mathcal
 D\subseteq H^\infty$, whence $u_i^*u_ia_n\in H^\infty$ for each $n\ge 1$. Combining with  the fact that $\alpha(u^*_i u_ia_n- u^*_i
 a)\le \alpha(u_ia_n-a)\rightarrow
 0$, we obtain that $u_i^*a\in H^\alpha$. Again from the choice of
 $u_i$, we know that $u_iu_i^*u_ia_n=u_ia_n$ for each $n\ge 1$. This
 implies that $a=u_i(u_i^*a)\in u_iH^\alpha$. Thus we conclude that $[u_{i}H^{\infty}]_{\alpha}\subseteq
 u_{i}H^{ \alpha},$ whence $[u_{i}H^{\infty}]_{\alpha}=
 u_{i}H^{ \alpha}.$

Now  from parts (3) and (4) of Theorem  \ref{invariant on L^alpha}
and from the definition of   internal column sum,  it follows that
\[\begin{aligned}
\mathcal W  &= [\mathcal W\cap \mathcal M]_\alpha  =\left [ \overline{span\{
\mathcal Z _1, u_iH^\infty : i\in \mathcal I \}}^{w*}\right ]_\alpha  =\left [ span\{
\mathcal Z _1, u_iH^\infty : i\in \mathcal I \}\right ]_\alpha \\&=\left
[ span\{ \mathcal Z , u_iH^\alpha : i\in \mathcal I \}\right
]_\alpha = \mathcal Z
\bigoplus^{col}(\bigoplus^{col}_{i}u_{i}H^{\alpha}). \end{aligned}\]

Next, we will verify that $  \mathcal Z =[ \mathcal Z
H_0^\infty]_{\alpha}. $ Recall that $ \mathcal Z  =[ \mathcal Z
_1]_{\alpha}$. It follows from part (1) of Theorem \ref{invariant on
L^alpha} we have that
$$
[ \mathcal Z _1H_0^\infty]_\alpha\cap \mathcal M =\overline{
\mathcal Z _1H_0^\infty}^{w*}= \mathcal Z _1.
$$Hence from part (3) of Theorem  \ref{invariant on L^alpha} we have
that
$$
 \mathcal Z \supseteq [ \mathcal Z H_0^\infty]_{\alpha}\supseteq [ \mathcal Z _1H_0^\infty]_{\alpha}=
[[ \mathcal Z _1H_0^\infty]_{\alpha} \cap \mathcal M]_\alpha =[
\mathcal Z _1]_\alpha= \mathcal Z .
$$ Thus
$
 \mathcal Z =[ \mathcal Z H_0^\infty]_{\alpha}. $

Moreover, it is not hard to verify that  for each $i,$   left
multiplication by the $u_{i}u_{i}^{*}$ are contractive projections
from $\mathcal W$ onto the summands $u_{i} H^{\alpha},$ and left
multiplication by $I-\sum_{i}u_{i}u_{i}^{*}$ is a contractive
projection from $\mathcal W$ onto $ \mathcal Z .$

Now the proof is completed.
 \end{proof}

A quick application of Theorem \ref{mainthm} is the following
corollary on doubly invariant subspaces in
$L^{\alpha}(\mathcal{M},\tau) $.
\begin{corollary} \label{Cor5.8}
Let $\mathcal M$ be a finite von Neumann algebra with a faithful,
normal, tracial state $\tau$.
 Let $\alpha$ be a normalized, unitarily
invariant, $\|\cdot \|_1$-dominating, continuous norm on $\mathcal
M$. If $\mathcal W$ is a   closed subspace of
$L^{\alpha}(\mathcal{M},\tau) $ such that $\mathcal W\mathcal
M\subseteq \mathcal W$, then there exists a projection $e$ in
$\mathcal M$ such that $\mathcal W=eL^\alpha(\mathcal M,\tau)$.
\end{corollary}
\begin{proof}
Note that $\mathcal M$ itself is a finite, maximal subdiagonal
subalgebra
  of $\mathcal M$. Let $H^\infty=\mathcal M$. Then $\mathcal
  D=\mathcal M$ and $\Phi$ is the identity map from $\mathcal M$ to
  $\mathcal M$. Hence $H_0^\infty=\{0\}$ and $H^\alpha=
  L^\alpha(\mathcal M,\tau)$.

Assume that $\mathcal W$ is a   closed subspace of
$L^{\alpha}(\mathcal{M},\tau) $ such that $\mathcal W\mathcal
M\subseteq \mathcal W$. From Theorem \ref{mainthm}, $$\mathcal W=
\mathcal Z \bigoplus^{col}(\bigoplus^{col}_{i\in \mathcal
I}u_{i}H^{\alpha}),$$ where $ \mathcal Z $ is a closed    subspace
of $L^{\alpha}(\mathcal{M},\tau)$ such that $ \mathcal Z =[ \mathcal
Z H_{0} ^{\infty}]_{\alpha}$ , and where $u_{i}$ are partial
isometries in ${\mathcal{M}}\cap \mathcal W$ with $u_{j}^{*}u_{i}=0$
if $i\neq j,$ and with $u_{i}^{*}u_{i}\in{\mathcal{D}}.$ Moreover, for each $i,$ $u_{i}%
^{*} \mathcal Z =\{0\},$ left multiplication by the $u_{i}u_{i}^{*}$
are contractive projections from $\mathcal W$ onto the summands
$u_{i} H^{\alpha},$ and left multiplication by
$I-\sum_{i}u_{i}u_{i}^{*}$ is a contractive projection from
$\mathcal W$ onto $ \mathcal Z .$

From the fact that $H_0^\infty=\{0\}$, we know that $ \mathcal Z
=\{0\}$. Since $\mathcal
  D=\mathcal M$, we know that
  $$
u_iH^\alpha= u_iL^\alpha(\mathcal M,\tau)\supseteq
u_iu^*_iL^\alpha(\mathcal M,\tau) \supseteq
u_iu^*_iu_iL^\alpha(\mathcal M,\tau)= u_iL^\alpha(\mathcal
M,\tau)=u_iH^\alpha.
  $$ So $
u_iH^\alpha=  u_iu^*_iL^\alpha(\mathcal M,\tau)$ and
$$\mathcal
W= \mathcal Z \bigoplus^{col}(\bigoplus^{col}_{i\in \mathcal
I}u_{i}H^{\alpha})= \bigoplus^{col}_{i\in \mathcal I}
u_iu^*_iL^\alpha(\mathcal M,\tau)= (\sum_i
u_iu^*_i)L^\alpha(\mathcal M,\tau)= eL^\alpha(\mathcal M,\tau),$$
where $e= \sum_i u_iu^*_i$ is a projection in $\mathcal M$.
\end{proof}

Next result is another application of Theorem \ref{mainthm} on
simply invariant subspaces in weak* Dirichlet algebras.
\begin{corollary}\label{Cor5.9}
Let $\mathcal M$ be a finite von Neumann algebra with a faithful,
normal, tracial state $\tau$. Let $H^\infty$  be a finite, maximal
subdiagonal subalgebra
  of $\mathcal M$  such that $   H^\infty\cap (H^\infty)^*=\mathbb CI.$
 Let $\alpha$ be a normalized, unitarily
invariant, $\|\cdot \|_1$-dominating, continuous norm on $\mathcal
M$.

Assume that $\mathcal W$ is a   closed subspace of
$L^{\alpha}(\mathcal{M},\tau).$  Then

\begin{enumerate}
\item if $\mathcal W$ is  simply
$H^\infty$-right invariant, i.e. $\mathcal WH^{\infty}\subsetneqq
\mathcal W$, then $\mathcal W=uH^\alpha$ for some   unitary $u\in \mathcal W\cap
\mathcal M$.
\item if $\mathcal W$ is   simply
$H^\infty$-right invariant in $H^\alpha$, i.e. $\mathcal WH^{\infty}\subsetneqq
\mathcal W$, then $\mathcal W=uH^\alpha$ with $u$ an inner element (i.e., $u$ is unitary and $u\in H^\infty$ ).\end{enumerate}
\end{corollary}

\begin{proof} It is not hard to see that part (2) follows directly from
part (1). We will need only to prove (1).  From Theorem \ref{mainthm},
$$\mathcal W= \mathcal Z \bigoplus^{col}(\bigoplus^{col}_{i\in
\mathcal I}u_{i}H^{\alpha}),$$ where $ \mathcal Z $ is a closed
subspace of $L^{\alpha}(\mathcal{M},\tau)$ such that $ \mathcal Z =[
\mathcal Z H_{0} ^{\infty}]_{\alpha}$ , and where $u_{i}$ are
partial isometries in ${\mathcal{M}}\cap \mathcal W$ with
$u_{j}^{*}u_{i}=0$
if $i\neq j,$ and with $u_{i}^{*}u_{i}\in    H^\infty\cap (H^\infty)^*.$ Moreover, for each $i,$ $u_{i}%
^{*} \mathcal Z =\{0\},$ left multiplication by the $u_{i}u_{i}^{*}$
are contractive projections from $\mathcal W$ onto the summands
$u_{i} H^{\alpha},$ and left multiplication by
$I-\sum_{i}u_{i}u_{i}^{*}$ is a contractive projection from
$\mathcal W$ onto $ \mathcal Z .$

Since $\mathcal WH^{\infty}\subsetneqq \mathcal W$, $
\bigoplus^{col}_{i\in \mathcal I}u_{i}H^{\alpha} \ne \{0\}$.
Therefore there exists some $i\in \mathcal I$ such that $u_i\ne 0$.
Then $u_{i}^{*}u_{i}$ is a nonzero projection in $
H^\infty\cap (H^\infty)^*=\mathbb CI$, or $u_{i}^{*}u_{i}=I.$ This
implies that $u_i$ is a unitary element in $\mathcal W\cap \mathcal
M$. From the choice of $\{u_i\}_{i\in\mathcal I}$, we further
conclude that $\mathcal W= u_iH^\alpha$.
\end{proof}

\end{document}